\documentclass[1p,12pt]{article}
\usepackage{graphicx}
\usepackage{amsmath,amsxtra,amssymb,latexsym, amscd,amsthm}

\usepackage[colorlinks=true,citecolor=black,linkcolor=black,urlcolor=blue]{hyperref}

\theoremstyle{plain}
\newtheorem{theorem}{Theorem}[section]
\newtheorem{corollary}[theorem]{Corollary}
\newtheorem{lemma}[theorem]{Lemma}

\theoremstyle{definition}

\theoremstyle{remark}
\newtheorem{remark}{Remark}

\numberwithin{equation}{section}
\numberwithin{figure}{section}
\numberwithin{table}{section}

\newcommand{\wt}{\operatorname{wt}}

\newcommand{\M}{\operatorname{M}}

\newcommand{\T}{\operatorname{T}}
\newcommand{\de}{\operatorname{de}}

\begin{document}
\setlength{\baselineskip}{13truept}
\title{New aspects of regions whose tilings are enumerated by perfect powers}
\author{TRI LAI\\
Indiana University\\
Department of Mathematics\\
 Bloomington, IN 47405, USA}
%\email{tmlai@indiana.edu}
\date{\today}

\maketitle
\begin{abstract}
In 2003, Ciucu presented a unified way to enumerate tilings of lattice regions by using a certain Reduction Theorem (Ciucu, Perfect Matchings and Perfect Powers, \emph{Journal of Algebraic Combinatorics}, 2003). In this paper we continue this line of work by investigating new families of lattice regions whose tilings are enumerated by perfect powers or products of several perfect powers. We prove a multi-parameter generalization of Bo-Yin Yang's theorem on fortresses (B.-Y. Yang, Ph.D. thesis, Department of Mathematics, MIT, MA, 1991).  On the square lattice with zigzag paths, we consider two particular families of regions whose numbers of tilings are always a power of 3 or twice a power of 3. The latter result provides a new proof for a conjecture of Matt Blum first proved by Ciucu. We also consider several new lattices obtained by periodically applying two simple subgraph replacement rules to the square lattice. On some of those lattices, we get new families of regions whose numbers of tilings  are given by products of several perfect powers. In addition, we prove a simple product formula for the number of tilings of a certain family of regions on a variant of the triangular lattice.

  % keywords are optional
\end{abstract}

\section{Introduction}
\label{intro}

Given a lattice in the plane, a (lattice) \textit{region} is a finite connected union of elementary regions of that lattice. A \textit{tile} is the union of any two elementary regions sharing an edge. A \textit{tiling} of the region $R$ is a covering of $R$ by tiles so that there are no gaps or overlaps. We denote by $\T(R)$ the number of tilings of region $R$. The \textit{dual graph} of a region $R$ is the graph whose vertices are the elementary regions of $R$, and whose edges connect two elementary regions precisely when they share an edge.

We consider only undirected finite graphs without loops, however, multiple edges are allowed. A \textit{perfect matching} of a graph $G$ is a collection of edges such that each vertex of $G$ is incident to precisely one edge in the collection. If the edges of $G$ have weights on them, $\M(G)$ denotes the sum of the weights of all perfect matchings of $G$, where the weight of a perfect matching is the product of the weights on its constituent edges. We call $\M(G)$ the \textit{matching generating function} of $G$. One easily sees that when $G$ has all edges weighted by 1, $\M(G)$ is exactly the number of perfect matchings of $G$.

By a well-known bijection between the tilings of a region $R$ and the perfect matchings of its dual graph $G$, we have $\T(R)=\M(G)$.
\begin{figure}\centering
\includegraphics[width=0.6\textwidth]{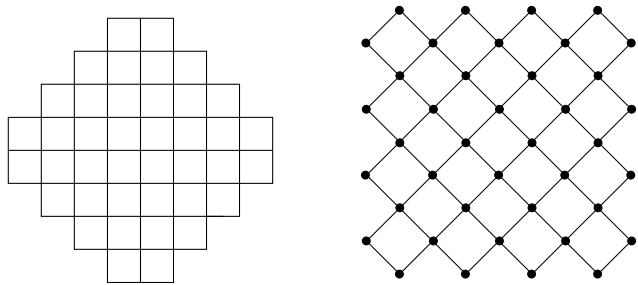}
\caption{The Aztec diamond region of order 4 (left) and  its dual graph (after rotated $45^o$), the Aztec diamond graph of order 4.}
\label{Fig1}
\end{figure}
%%Figures in this paper were made by WinFIG, Version 4.8, August-2011, by Andreas Schmidt
%
In the early 1990's, Elkies, Kuperberg, Larsen and Propp \cite{Elkies} considered a family of simple regions on the square lattice called Aztec diamonds (see Figure \ref{Fig1} for an example), and proved that the number of domino tilings of the Aztec diamond of order $n$ is $2^{n(n+1)/2}$.

A large body of related work followed (see for example \cite{Ciucu5},  \cite{Doug}, \cite{Propp}, \cite{Yang}),  centered on families of lattice regions whose tilings are enumerated by perfect powers or near perfect powers. In 2003, Ciucu \cite{Ciucu5} presented an approach that allows finding the number of tilings of such families of regions in a unified way. In particular, to find the number of tilings of a region, we find the number of perfect matchings of its dual graph. Then we deform the dual graph into a weighted Aztec diamond graph by using some simple subgraph replacement rules, and find matching generating function of the resulting weighted graph. We encode the weights of edges in the weighted Aztec diamond graph by a certain matrix, and then apply  repeatedly a naturally arising operator $\de$ to this matrix. Finally, after several applications of $\de$, we get a new matrix of a same type as the original one. This together with Reduction Theorem (called Generalized Domino-Shuffling in \cite{propp}) yield simple recurrences that determine the matching generating function.

In Section 3, we use this approach to prove a new multi-parameter generalization of Bo-Yin Yang's theorem \cite{Yang} on fortress regions. We show that the number tilings of a generalized fortress is always a product of a power of 2 and a power of 5  (see Theorem \ref{maineven}). We also prove a new counterpart of Stanley's multi-parameter generalization (see \cite{Ciucu2}) of Aztec diamond theorem \cite{Elkies}.

Section 4 investigates two new families of regions on the square lattice with every second zigzag path drawn in. We prove that the numbers of tilings in this case are either a power of 3 or twice a power of 3. Ciucu \cite{Ciucu5} also proved  a conjecture of Blum that the number of perfect matchings of a certain family of subgraphs of the square lattice is a power of $3$ or twice a power of $3$. The two formulas are nearly identical. It turns out that one can establish a direct connection between them. This provides an unexpected new proof for the conjecture.

We give a unified way to create new lattices from the square lattice in Section 5. We consider two special subgraph replacement rules for the nodes of the square lattice.  Periodically applying these rules gives us a large number of new lattices. On those lattices, we investigate several families of regions that are similar to Aztec diamonds or fortresses. In some cases, their numbers of tilings are given by products of several perfect powers.

Finally in Section 6, we create a simple variant of the standard triangular lattice by periodically removing some lattice segments. We consider a ``rhombus-shaped" region on the resulting lattice that has the number of tilings given by a product of a power of 2 and a power of 3.

%%%%%%%%%%%%%%%%%%%%%%%%%%%%%%

\section{Preliminary results and Reduction Theorem}
\label{sec2}

Before going to the statement of Reduction Theorem, we employ several basic preliminary results stated below.

A \textit{forced edge} of a graph is an edge contained in every perfect matching of $G$. Assume that $G$ is a weighted graph with weight function $\wt$ on its edges, and $G'$ is obtained from $G$ by removing forced edges $e_1,\dotsc,e_n$, and removing the vertices incident to these forced edges. Then one clearly has
\begin{equation}
\M(G)=\M(G')\prod_{i=1}^n\wt(e_i).
\end{equation}
\textit{From now on, whenever we remove some forced edges, we remove also the vertices incident to them.}

\begin{figure}\centering
\includegraphics[width=8cm]{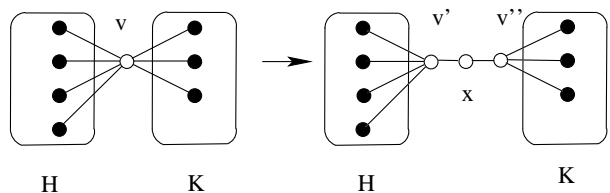}
\caption{Vertex splitting.}
\label{Fig2}
\end{figure}

\begin{lemma} [Vertex-splitting Lemma]\label{VS}
 Let $G$ be a graph, $v$ be a vertex of it, and denote the set of neighbors of $v$ by $N(v)$.
  For an arbitrary partition $N(v)=H\cup K$, let $G'$ be the graph obtained from $G\setminus v$ by including three new vertices $v'$, $v"$ and $x$ so that $N(v')=H\cup \{x\}$, $N(v")=K\cup\{x\}$, and $N(x)=\{v',v"\}$ (see Figure \ref{Fig2}). Then $\M(G)=\M(G')$.
\end{lemma}

\begin{lemma}[Edge-replacing Lemma]\label{ER}
Let $G$ be a weighted graph with weight function $\wt$ on its edges, $u$ and $v$ be two distinct vertices of $G$. Assume that $e_1$ and $e_2$ are two edges connecting $u$ and $v$. Let $G'$ be the graph obtained from $G$ by replacing two edges $e_1$ and $e_2$ by a new edge of weight $\wt(e_1)+\wt(e_2)$ that connects $u$ and $v$. Then $M(G)=M(G')$.
\end{lemma}

\begin{lemma}[Star Lemma]\label{star}
Let $G$ be a weighted graph, and let $v$ be a vertex of~$G$. Let $G'$ be the graph obtained from $G$ by multiplying the weights of all edges that are incident to $v$ by $t>0$. Then $\M(G')=t\,\M(G)$.
\end{lemma}

%\begin{proof}
%Since we only change the weights of some edges, the set of perfect matchings is unchanged. If $G$ does not have any perfect matching, then $\M(G)=\M(G')=0$. Assume that $G$ has a perfect matching. Any perfect matching $\mu$ of $G$ has exactly one edge that is incident to $v$. Thus, only one edge in the perfect matching $\mu$ of $G$ is multiplied by $t$. It implies that the weight of  $\mu$ in $G'$ is $t\,\omega_G(\mu)$, where $\omega_G(\mu)$ denotes the weight of perfect matching $\mu$ in graph $G$.  Therefore, $\M(G')=\sum_{\mu}\omega_{G'}(\mu)=\sum_{\mu}t\,\omega_G(\mu)=t\M(G)$.
%  \end{proof}

Part (a) of the following result is a generalization due to Propp of the ``urban renewal" trick first observed by Kuperberg. Parts (b) and (c) are due to Ciucu (see Lemma 2.6 in [5]).
\begin{figure}\centering
\includegraphics[width=0.5\textwidth]{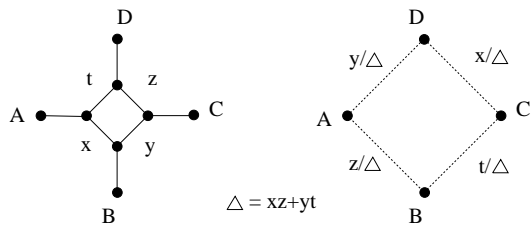}
\caption{Urban renewal.}
\label{Fig3}
\end{figure}
\begin{figure}\centering
\includegraphics[width=0.85\textwidth]{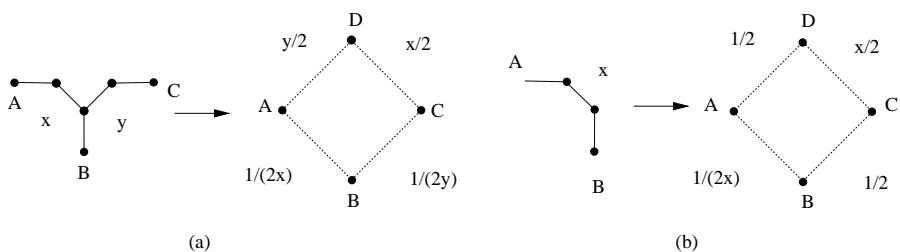}
\caption{Two variants of urban renewal.}
\label{Fig4}
\end{figure}

\begin{lemma} [Spider Lemma]\label{spider}
(a) Let $G$ be a weighted graph containing the subgraph $K$ shown on the left in Figure \ref{Fig3} (the labels indicate weights, unlabeled edges have weight 1). Suppose in addition that the four inner black vertices in the subgraph $K$, different from $A,B,C,D$, have no neighbors outside $K$. Let $G'$ be the graph obtained from $G$ by replacing $K$ by the graph $\overline{K}$ shown on right in Figure
\ref{Fig3}, where the dashed lines indicate new edges, weighted as shown. Then $\M(G)=(xz+yt) \M(G')$.

(b) Consider the above local replacement operation when $K$ and $\overline{K}$ are graphs shown in Figure \ref{Fig4}(a) with the indicated weights (in particular, $K'$ has a new vertex $D$ that is incident only to $A$ and $C$). Then $\M(G)=2\M(G')$.

(c) The statement of part (b) is also true when $K$ and $\overline{K}$ are the graphs indicated in Figure \ref{Fig4}(b) (in this case $G'$ has two new vertices $C$ and $D$ that are adjacent only to one another and to $B$ and $A$, respectively).
\end{lemma}

The centers of the edges of the Aztec diamond graph of order $n$, denoted by $AD_n$, form an $2n\times 2n$ array. The entries of this array are the weights of these edges. We call this $2n\times 2n$  array \textit{the weight matrix} of the weighted Aztec diamond $AD_n$. We are interested in the case of periodic weight matrix.

Let $A$ be an $k\times l$ matrix with $k$ and $l$ even. Place a copy of $A$ in the upper left corner of the weight matrix and fill
in the rest of the array periodically with period $A$ (i.e. translate $A$ to the right $l$ units at a time and  down $k$ units at a time; if the size of the weight matrix is not a multiple of $k$ or $l$, some of these translates will fit only partially in the array). Define the weight $\wt_A$ on the edges of $AD_n$ by assigning each edge the corresponding entry of $A$ in the array described above. In this case, $A$ is called the \textit{weight pattern} of the weighted Aztec diamond. Denote by $AD_n(\wt_A)$ the Aztec diamond graph of order $n$ with the weight pattern $A$.

The following useful lemma was first proved by Ciucu.

\begin{lemma}\label{mulstar}
(a) (Lemma 4.4 in \cite{Ciucu5}) Consider the Aztec Diamond graph of order $n$ with $2n\times 2n$ weight matrix $D$. We divide $D$ into $n+1$ parts: the first column (resp., row), the last column (resp., row), and $2k$-th and $(2k+1)$-th columns (resp., rows), for $k=1,\dots,n-1$. $D'$ is the matrix obtained from $D$ by multiplying all entries of some part by a positive number $t>0$, then
\begin{equation}
\M(AD_n(\wt_D))=t^{-n}\M(AD_n(\wt_{D'})).
\end{equation}

(b) (Lemma 6.2 in \cite{Ciucu5}) We now divide matrix $D$ above into $n$ parts, so that the $i$-th part consists of the $(2i-1)$-th and $2i$-th columns (resp., rows), for $i=1,2,\dots,n$. $D''$ is the matrix obtained from $D$ by multiplying all entries of some part by a positive number $t$, then
\begin{equation}
\M(AD_n(\wt_D))=t^{-n-1}\M(AD_n(\wt_{D''})).
\end{equation}
\end{lemma}

We consider a consequence of the Star Lemma \ref{star}.

\begin{lemma}\label{matrixstar}
Consider the Aztec diamond graph of order $n$ with $2n\times 2n$ weight matrix $D=(m_{ij})_{1\leq i,j\leq 2n}$. We divide the matrix $D$ into $(n+1)\cdot n$ blocks $M_{ij}$ defined as follows.
\begin{enumerate}
 \item $M_{1,j}:=\begin{bmatrix}m_{1,2j-1}&m_{1,2j}\end{bmatrix}$, for $j=1,\dots,n$;

\item $M_{i,j}:=\begin{bmatrix}m_{2i-2,2j-1}&m_{2i-2,2j}\\m_{2i-1,2j-1}&m_{2i-1,2j}\end{bmatrix}$, for $i=2,\dots,n$ and $j=1,2,\dots,n$;

\item $M_{n+1,j}:=\begin{bmatrix}m_{2n,2j-1}& m_{2n,2j}\end{bmatrix}$, for $j=1,\dots,n$.
\end{enumerate}
Let $D'$ be the matrix obtained from $D$ by multiplying all entries of some block $M_{i,j}$ by $t>0$, then
\begin{equation}
\M(AD_n(\wt_{D'}))=t\,\M(AD_n(\wt_D)).
\end{equation}
\end{lemma}

For a $k\times l$ matrix $A$ with $k$ and $l$ even we define a new $k\times l$ matrix $\de(A)$ as follows. Divide matrix $A$ into $2\times2$ blocks
\[
\begin{bmatrix}
x &w\\
y &z
\end{bmatrix},
\]
and assume $xz+yw\not=0$ for all such blocks. Replace each such block by the following block
\[
\begin{bmatrix}
z/(xz+yw) &y/(xz+yw)\\
w/(xz+yw) &x/(xz+yw)
\end{bmatrix}.
\]
We get a new  $k\times l$ matrix, denoted by $B$. Define $\de(A)$ to be the $k\times l$ matrix obtained from $B$ by cyclically shifting its columns one unit up and cyclically shifting the rows of resulting matrix one unit left.

\begin{figure}\centering
\includegraphics[width=0.4\textwidth]{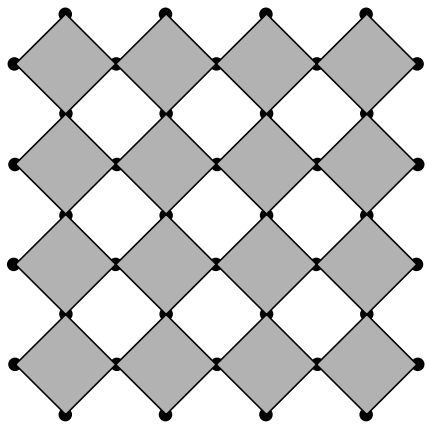}
\caption{The edges of $AD_4$ are partitioned into $4^2$ cells.}
\label{Fig5}
\end{figure}

The edges of the Aztec diamond graph $AD_n$ can be partitioned into $n^2$ $4$-cycles, which we call the \textit{cells}\footnote{The definition of cells and cell-factors above has been used for a more general family of graphs, named \textit{cellular graphs}. One can see \cite{Ciucu2}, \cite{Ciucu5} and \cite{Ciucu4} for more details.} of the graph, such that each vertex is contained in at most two cells (see the shaded diamonds in Figure \ref{Fig5}). The Aztec diamond graph $AD_n$ has $n$ rows and $n$ columns of cells. If the cell $c$ has edges weighted by $x,$ $y,$ $z,$ $t$ (in cyclic order), then the \textit{cell-factor} $\Delta(c)$ of $c$ is defined by setting $\Delta(c):=xz+yt$.

Assume that the Aztec diamond graph of order $0$ has matching generating function 1.  We have the following Reduction Theorem due to Propp. \cite{propp}

\begin{theorem}[Reduction Theorem]\label{reduction}
Assume that the cells of $AD_n(\wt_A)$ have nonzero cell-factors. Then
\begin{equation}
\M(AD_n(\wt_A))= \M(AD_{n-1}(\wt_{\de(A)})) \prod_c\Delta(c),
\end{equation}
where the product is taken over all cells $c$ of $AD_n(\wt_A)$.
\end{theorem}

With the assumption that all cell-factors are nonzero in $AD_{n-i}(\wt_{\de^{(i)}(A)})$, for $i=0,1,2,\dotsc, n$, we can apply Theorem \ref{reduction} consecutively until we get down the Aztec diamond of order $0$. In particular cases, the weight pattern repeats or changes in a simple predictable way after a small number of successive applications. This provides some recurrences, and the matching generating function of the original weighted Aztec diamond can be obtained recursively.

%%%%%%%%%%%%%%%%%%%%%%%%%%%%%
\section{Generalization of fortress regions}
\label{sec3}

Yang  \cite{Yang} showed that the number of tilings of  a \textit{fortress} (called \emph{Penta-Aztec-Diamond} in \cite{Yang}; see the left picture in Figure \ref{Fig6} for an example) on the square lattice with all diagonals drawn in is always  a power of $5$ or twice a power of 5. In particular, the number of tilings of the fortress of order $n$, denoted by $F_n$, is obtained by the following theorem.

\begin{figure}\centering
\includegraphics[width=0.80\textwidth]{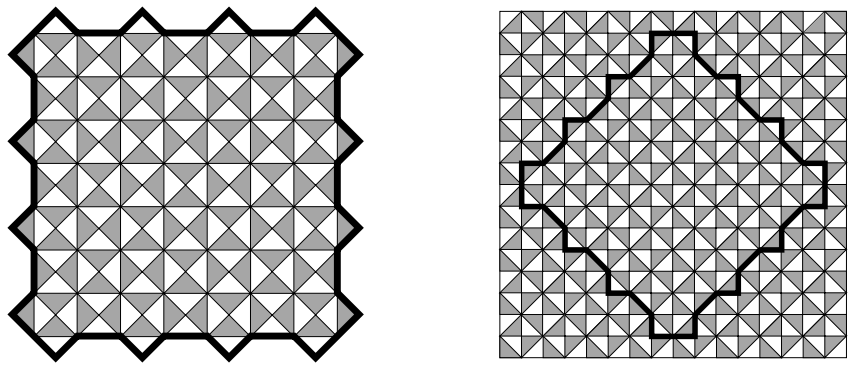}
\caption{The fortress of order 7 on two different lattices.}
\label{Fig6}
\end{figure}

\begin{theorem}[Theorem 3.1. in \cite{Yang}]
\begin{equation}
\T(F_n)=
\begin{cases}
5^{k^2} &\text{if $n=2k$;}\\
5^{2k(2k+1)} &\text{if $n=4k+1$;}\\
2\cdot5^{2k(2k-1)} &\text{if $n=4k-1$.}\\
\end{cases}
\notag\end{equation}
\end{theorem}

A fortress (after rotated $45^o$) can be viewed also as  a region on the square lattice with all \textit{second} diagonals drawn in. The vertices of the fortress of order $n$ are now the vertices of a diamond of side-length $n\sqrt{2}$ (see the right picture in Figure \ref{Fig6}).

We consider next a generalization of the fortresses defined as follows. Assume $d_1,$ $d_2,$ $\dotsc,$  $d_m$ are positive integers. Draw in all second southwest-to-northeast diagonals. Then draw in $m+1$ southeast-to-northwest diagonals, with the distances between successive ones, starting from bottom, being $d_1\sqrt{2},$ $d_2\sqrt{2},\dotsc,d_m\sqrt{2}$. Assume in addition that each of the $m+1$ southeast-to-northwest diagonals intersects each of the second southwest-to-northeast diagonals at a lattice point of the square lattice.

Pick two lattice points $A$ and $B$ on the bottom southeast-to-northwest diagonal drawn in, and pick two lattice points $C$ and $D$ on the top southeast-to-northwest diagonal drawn in, so that $A,B,C,D$ are four vertices of a diamond of side-length $(d_1+d_2+\dotsc+d_m)\sqrt{2}$ in cyclic order.

Color the resulting dissection of the square lattice black and white so that any two elementary regions that share an edge have opposite color. Without loss of generality, we assume that the triangular elementary region, which has two edges resting on the segments $AD$ and $AC$, is white.

Start from $A$ and take unit steps south, east or southeast so that for each step the color of the elementary region on the left is white. This path ends when reaching $B$. The described path from $A$ to $B$ is the southwestern boundary of the region. We get the southeastern boundary by going from $B$ to $C$ in similar fashion with  unit steps north, east or northeast so that the elementary region on the left is black for each step. The northeastern boundary is obtained by reflecting the southeastern boundary about the line passing $A$ and $C$; and the northwestern boundary is obtained by reflecting the southwestern boundary about the line passing $A$ and $C$ (see Figure \ref{Fig7} for examples).

\begin{figure}\centering
\includegraphics[width=0.75\textwidth]{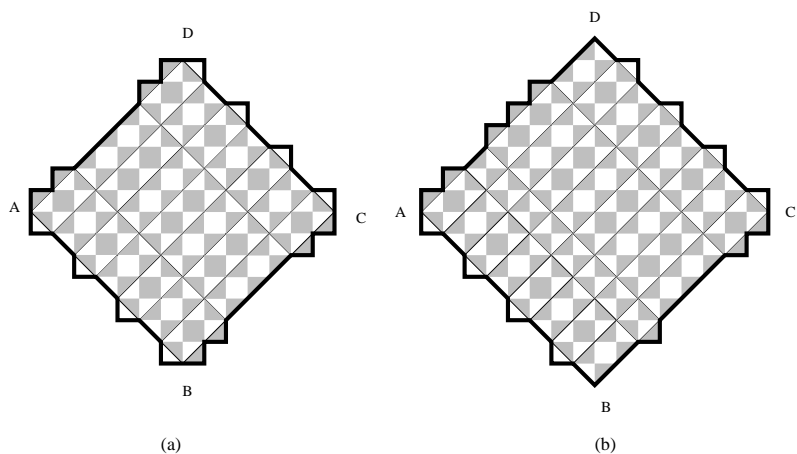}
\caption{Two regions  $F(2,3,2)$ (a) and $F(2,1,3,2)$ (b).}
\label{Fig7}
\end{figure}

Denote by $F(d_1,d_2,\dots, d_m)$ the region bordered by four lattice paths above; we call it a \textit{generalized fortress}. When $d_1=d_2=\dotsc=d_m=1$, we get the original fortress of order $m$.

We are also interested in a variant of the generalized fortress $F(d_1,\dots,d_m)$ defined as follows. Repeat the whole process in the definition of $F(d_1,\dotsc,d_m)$ above, with the one change that on the southwestern boundary we make the switch from the rule ``white on left'' to ``black on left''; and on the southeastern boundary we make the switch from the rule ``black on left'' to ``white on left''. Denote by $\overline{F}(d_1,d_2,\dots, d_m)$ the resulting region (see Figure \ref{Fig8}  for examples).

\begin{figure}\centering
\includegraphics[width=0.75\textwidth]{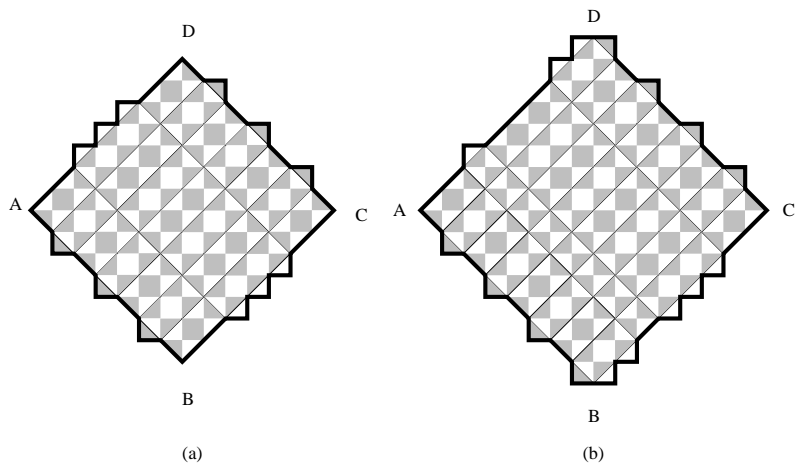}
\caption{Two regions  $\overline{F}(2,3,2)$ (a) and $\overline{F}(2,1,3,2)$ (b).}
\label{Fig8}
\end{figure}

The number of tilings of the generalized fortresses is given by the following theorem.

\begin{theorem}\label{main}Let $d_1,\dotsc,d_m$ be positive integers. Let $s_j:=d_1+d_2+\dotsc+d_j$ ($s_0:=0$)  and $S:=\sum_{j=1}^m\min(s_j, s_m-s_j)$, for $j=1,2,\dots,m$.

(a) If  $d_1+\dots+d_m=2k$, then
\begin{equation}\label{maineven}
\T(F(d_1,d_2,\dotsc,d_m))=\T(\overline{F}(d_1,d_2,\dotsc,d_m))= 2^{2k^2-2S}\, 5^{S}.
\end{equation}

(b) If $d_1+\dots+d_m=2k+1$, then
\begin{equation}\label{mainodd1}
\T(F(d_1,d_2,\dots,d_m))= 2^{\alpha}\, 5^{S},
\end{equation}
\begin{equation}\label{mainodd2}
\T(\overline{F}(d_1,d_2,\dotsc,d_m))=2^{(2k+1)^2-4S-\alpha}\, 5^{S},
\end{equation}

where
\begin{equation}
\alpha=\begin{cases}
k(2k+2)-2S &\text{if $s_{2j}<k+1\leq s_{2j+1}$ for some $j$;}\\
k(2k+2)+1-2S &\text{ otherwise.}
\end{cases}
\notag\end{equation}

\end{theorem}

Before going to the proof of Theorem \ref{main}, we present a number of results about weighted Aztec diamonds with multi-parameter weight pattern, which we will employ in the proof of Theorem \ref{main}.

Stanley found that some periodic weights of the Aztec diamond give a simple product formula for the matching generating function (see \cite{Ciucu2} or Section 2.3 in \cite{Yang}). Recall that we denote by $AD_n(\wt_S)$ the Aztec diamond graph of order $n$ with weight pattern $S$.
\begin{theorem}[Stanley]
\begin{equation}\label{stanleyformula}
\M(AD_n(\wt_S))=\prod_{1\leq i\leq j\leq n }(x_iw_j+y_it_j),
\end{equation}
where
\begin{equation}
S=\begin{bmatrix}
x_1&y_1&x_2&y_2&\dots&x_{n-1}&y_{n-1}&x_n&y_n\\
t_1&w_1&t_2&w_2&\dots&t_{n-1}&w_{n-1}&t_n&w_n
\end{bmatrix}.
\notag\end{equation}
\end{theorem}
There are several related results due to B.-Y. Yang \cite{Yang}, and due to Ciucu (Corollary 4.3 and 5.5 in \cite{Ciucu5}).

In the spirit of Stanley's formula (\ref{stanleyformula}), we prove next a simple product formula for the matching generating function of the Aztec diamond with weight pattern
\begin{equation}A=
\begin{bmatrix}
a_1&b_1&a_2&b_2&\dotsc&a_{n-1}&b_{n-1}&a_n&b_n\\
a_1&b_1&a_2&b_2&\dotsc&a_{n-1}&b_{n-1}&a_n&b_n\\
c_1&d_1&c_2&d_2&\dotsc&c_{n-1}&d_{n-1}&c_n&d_n\\
c_1&d_1&c_2&d_2&\dotsc&c_{n-1}&d_{n-1}&c_n&d_n
\end{bmatrix},
\notag\end{equation}
where $a_i$, $b_i$, $c_i$ and $d_i$ are positive numbers, for $i=1,2\dotsc,n$.

\begin{theorem}\label{weighted}
(a) If $n=2k$, then
\begin{align}\label{ge}
\M(AD_n(\wt_{A}))=&2^{k(k+1)}\prod_{i=1}^{k}(a_ib_{n-i+1}c_id_{n-i+1})^{k-i+1}\notag\\
&\times\prod_{j=1}^{k}\prod_{i=j}^{n-j}(d_{i}a_{i+1}+b_{i}c_{i+1}).
\end{align}

(b) If $n=2k+1$, then
\begin{align}\label{go}
\M(AD_n(\wt_{A}))=&2^{(k+1)^2}\prod_{i=1}^{k+1}(a_ib_{n-i+1})^{k-i+2}(c_id_{n-i+1})^{k-i+1}\notag\\
&\times\prod_{j=1}^{k}\prod_{i=j}^{n-j}(d_{i}a_{i+1}+b_{i}c_{i+1}).
\end{align}
\end{theorem}

\begin{proof}
One easily sees that
\[\de(A)=
\begin{bmatrix}
1/(2b_1)&1/(2a_2)&1/(2b_2)&1/(2a_3)&\dotsc&1/(2b_n)&1/(2a_1)\\
1/(2d_1)&1/(2c_2)&1/(2d_2)&1/(2a_3)&\dotsc&1/(2d_n)&1/(2c1)\\
1/(2d_1)&1/(2c_2)&1/(2d_2)&1/(2a_3)&\dotsc&1/(2d_n)&1/(2c1)\\
1/(2b_1)&1/(2a_2)&1/(2b_2)&1/(2a_3)&\dotsc&1/(2b_n)&1/(2a_1)
\end{bmatrix},
\]
and
\[
\de^{(2)}(A)=
\begin{bmatrix}
\frac{2\Theta_1}{b_1\Delta_1}&\frac{2\Theta_2}{a_3\Delta_2}&\frac{2\Theta_2}{b_2\Delta_2}&\dotsc&\frac{2\Theta_n}{a_1\Delta_n}&\frac{2\Theta_n}{b_n\Delta_n}&\frac{2\Theta_1}{a_2\Delta_1}\\
\frac{2\Theta_1}{b_1\Delta_1}&\frac{2\Theta_2}{a_3\Delta_2}&\frac{2\Theta_2}{b_2\Delta_2}&\dotsc&\frac{2\Theta_n}{a_1\Delta_n}&\frac{2\Theta_n}{b_n\Delta_n}&\frac{2\Theta_1}{a_2\Delta_1}\\
\frac{2\Theta_1}{d_1\Delta_1}&\frac{2\Theta_2}{c_3\Delta_2}&\frac{2\Theta_2}{d_2\Delta_2}&\dotsc&\frac{2\Theta_n}{c_1\Delta_n}&\frac{2\Theta_n}{d_n\Delta_n}&\frac{2\Theta_1}{c_2\Delta_1}\\
\frac{2\Theta_1}{d_1\Delta_1}&\frac{2\Theta_2}{c_3\Delta_2}&\frac{2\Theta_2}{d_2\Delta_2}&\dotsc&\frac{2\Theta_n}{c_1\Delta_n}&\frac{2\Theta_n}{d_n\Delta_n}&\frac{2\Theta_1}{c_2\Delta_1}
\end{bmatrix},
\]
where $\Delta_i=a_{i+1}d_i+b_ic_{i+1}$ and $\Theta_i=a_{i+1}b_ic_{i+1}d_i$, for $i=1,2,\dots,n$ (the subscripts here are interpreted modulo $n$).

\medskip
(a) Suppose that $n=2k$, for some positive integer $k$. The cells in the $i$-th column of the weighted Aztec diamond $AD_n(\wt_A)$ have cell-factors either $2a_ib_i$ or $2c_id_i$, for $i=1,2,\dots,n$.
Apply Reduction Theorem \ref{reduction}, we have
\begin{equation}\label{eq1}
\M(AD_n(\wt_A))=\M(AD_{n-1}(\wt_{\de(A)})) \prod_{i=1}^n(2a_ib_i)^k(2c_id_i)^k.
\end{equation}
The cells in the $i$-th column of the weighted Aztec diamond $AD_{n-1}(\wt_{d(A)})$ have cell-factors $\frac{\Delta_i}{4\Theta_i}$, for $i=1,2,\dots,n-1$. Thus, by the Reduction Theorem again, we obtain
\begin{equation}\label{eq2}
\M(AD_{n-1}(\wt_{\de(A)}))= \M(AD_{n-2}(\wt_{\de^{(2)}(A)})) \prod_{i=1}^{n-1}\left(\frac{\Delta_i}{4\Theta_i}\right)^{n-1}.
\end{equation}
Divide the weight matrix of the Aztec diamond $AD_{n-2}(\wt_{\de^{(2)}(A)})$  into $n-1$ parts consisting of columns as in Lemma \ref{mulstar}(a). Multiply all entries in the $i$-th part (from left to right)
 by $\frac{\Delta_i}{2\Theta_i}$, for $i=1,2,\dots,n-1$. We get the weight matrix of $AD_{n-2}(\wt_{B})$,
  where $B$ is the $4\times 2(n-2)$ matrix defined by
\[B:=
\begin{bmatrix}
1/b_1&1/a_3&1/b_2&1/a_4&\dots&1/b_{n-2}&1/a_n\\
1/b_1&1/a_3&1/b_2&1/a_4&\dots&1/b_{n-2}&1/a_n\\
1/d_1&1/c_3&1/d_2&1/c_4&\dots&1/d_{n-2}&1/c_n\\
1/d_1&1/c_3&1/d_2&1/c_4&\dots&1/d_{n-2}&1/c_n
\end{bmatrix}.
\]
Therefore, Lemma \ref{mulstar}(a) implies
\begin{equation}\label{eq3}
\M(AD_{n-2}(\wt_{\de^{(2)}(A)}))= \M(AD_{n-2}(\wt_{B}))\prod_{i=1}^{n-1}\left(\frac{2\Theta_i}{\Delta_i}\right)^{n-2}.
\end{equation}
We get the following recurrence by applying three equalities (\ref{eq1}), (\ref{eq2}) and (\ref{eq3})
\begin{equation}\label{eq4}
\frac{\M(AD_n(\wt_{A}))}{\M(AD_{n-2}(\wt_{B}))}=2^{n}\prod_{i=1}^na_i^kb_i^kc_i^kd_i^k\cdot\prod_{i=1}^{n-1}\frac{\Delta_i}{\Theta_i}.
\end{equation}

Next, we apply the recurrence (\ref{eq4}) to the weighted Aztec diamond $AD_{n-2}(\wt_{B})$ and obtain
\begin{align}\label{eq5}
\frac{\M(AD_{n-2}(\wt_{B}))}{\M(AD_{n-4}(\wt_{C}))}&=2^{n-2}\prod_{i=1}^{n-2}\left(\frac{1}{b_ia_{i+2}d_ic_{i+2}}\right)^{k-1}
\cdot\prod_{i=1}^{n-3}\left(\frac{\frac{1}{c_{i+2}b_{i+1}}+\frac{1}{a_{i+2}d_{i+1}}}{\frac{1}{b_{i+1}a_{i+2}d_{i+1}c_{i+2}}}\right)\notag\\ &=2^{n-2}\prod_{i=1}^{n-2}\left(\frac{1}{b_ia_{i+2}d_ic_{i+2}}\right)^{k-1}\cdot\prod_{i=2}^{n-2}\Delta_i,
\end{align}
where
\[
C=
\begin{bmatrix}
a_3&b_3&a_4&b_4&\dots&a_{n-2}&b_{n-2}\\
a_3&b_3&a_4&b_4&\dots&a_{n-2}&b_{n-2}\\
c_3&d_3&c_4&d_4&\dots&c_{n-2}&d_{n-2}\\
c_3&d_3&c_4&d_4&\dots&c_{n-2}&d_{n-2}
\end{bmatrix},
\]
i.e. $C$ is obtained from the matrix $A$ by removing the first and the last four columns.
Two equalities (\ref{eq4}) and (\ref{eq5}) imply
\begin{align}\label{eq6}
\frac{\M(AD_n(\wt_{A}))}{\M(AD_{n-4}(\wt_{C}))}=&2^{2n-2}(a_1b_nc_1d_n)^{k}(a_2b_{n-1}c_{2}d_{n-1})^{k-1}\notag\\
&\times\prod_{i=1}^{n-1}\Delta_i\cdot\prod_{i=2}^{n-2}\Delta_i.
\end{align}

We now consider an operator $\phi$ defined as follows. Let $N$ be an $m\times n$ matrix with $n\geq8$, then $\phi(N)$ is the $m\times (n-8)$ matrix obtained from $N$ by removing its first four columns and its last four columns. In particular, $\phi(A)=C$.

If $k=2q$, we apply the recurrence (\ref{eq6}) $q$ times and obtain
\begin{equation}\label{g1}
\M(AD_{4q}(\wt_{A}))=2^{2q(2q+1)}\prod_{i=1}^{k}(a_ib_{n-i+1}c_id_{n-i+1})^{k-i+1}\cdot\prod_{j=1}^{k}\prod_{i=j}^{n-j}\Delta_i.
\end{equation}
In the case $k=2q+1$, we apply also the recurrence (\ref{eq6}) $q$ times and get
\begin{equation}
\frac{\M(AD_{4q+2}(\wt_{A}))}{\M(AD_{2}(\wt_{\phi^{(q)}(A)}))}=2^{2q(2q+3)}\prod_{i=1}^{k-1}(a_ib_{n-i+1}c_id_{n-i+1})^{k-i+1}
\cdot\prod_{j=1}^{k-1}\prod_{i=j}^{n-j}\Delta_i,
\end{equation}
where
\[\phi^{(q)}(A)=\begin{bmatrix}
a_k&b_k&a_{k+1}&b_{k+1}\\
a_k&b_k&a_{k+1}&b_{k+1}\\
c_k&d_k&c_{k+1}&d_{k+1}\\
c_k&d_k&c_{k+1}&d_{k+1}
\end{bmatrix}.
\]
Moreover, by Reduction Theorem, we get easily that
\begin{equation}
\M(AD_{2}(\wt_{\phi^{(q)}(A)}))=2^2(a_kb_{k+1}c_kd_{k+1})(a_{k+1}d_k+b_{k}c_{k+1}).
\end{equation}
Therefore,
\begin{equation}\label{g2}
\M(AD_{4q+2}(\wt_{A}))=2^{2q(2q+3)+2}\prod_{i=1}^{k}(a_ib_{n-i+1}c_id_{n-i+1})^{k-i+1}\cdot\prod_{j=1}^{k}\prod_{i=j}^{n-j}\Delta_i.
\end{equation}
Finally, the equalities (\ref{g1}) and (\ref{g2}) yield (\ref{ge}).

\medskip
(b) Suppose that $n=2k+1$, for some nonnegative integer $k$. This case can be treated similarly to the case of even $n$. Two equations (\ref{eq2}) and (\ref{eq3}) are also true in this case. The exponents of $2a_ib_i$ and $2c_id_i$ in (\ref{eq1}) are now $k+1$ and $k$, respectively (as opposed to both being $k$ when $n=2k$). Thus, we have
\begin{equation}\label{eq1n}
\M(AD_n(\wt_A))=\M(AD_{n-1}(\wt_{\de(A)}))\,\prod_{i=1}^n(2a_ib_i)^{k+1}(2c_id_i)^k.
\end{equation}
By (\ref{eq1n}), (\ref{eq2}) and (\ref{eq3}), we get the following equation (instead of (\ref{eq4}) when $n=2k$)
\begin{equation}\label{eq4n}
\frac{\M(AD_n(\wt_{A}))}{\M(AD_{n-2}(\wt_{B}))}=2^{n}\prod_{i=1}^na_i^{k+1}b_i^{k+1}c_i^kd_i^k\cdot\prod_{i=1}^{n-1}\frac{\Delta_i}{\Theta_i}.
\end{equation}
We apply (\ref{eq4n}) twice and obtain an equality (instead of (\ref{eq6}) when $n=2k$) as follows.
\begin{align}\label{eq6n}
\frac{\M(AD_n(\wt_{A}))}{\M(AD_{n-4}(\wt_{\phi(A)}))}=&2^{2n-2}(a_1b_n)^{k+1}(a_2b_{n-1})^k(c_1d_n)^{k}(c_2d_{n-1})^{k-1}\notag\\
&\times\prod_{i=1}^{n-1}\Delta_i\cdot\prod_{i=2}^{n-2}\Delta_i.
\end{align}
Finally, we get (\ref{go}) by applying repeatedly  (\ref{eq6n}).
  \end{proof}

\begin{remark}
Ciucu \cite{Ciucu5} considered a similar multi-parameter weight pattern
\begin{equation}A'=
\begin{bmatrix}
x_1&x_2&\dotsc&x_{2n-1}&x_{2n}\\
y_1&y_2&\dotsc&y_{2n-1}&y_{2n}\\
y_1&y_2&\dotsc&y_{2n-1}&y_{2n}\\
x_1&x_2&\dotsc&x_{2n-1}&x_{2n}
\end{bmatrix},
\notag\end{equation}
and got two recurrences similar to (\ref{eq4}) and (\ref{eq4n}) (see Theorem 5.1 in \cite{Ciucu5}). However, he did not give an explicit product formula for the matching generating function of $AD_n(\wt_{A'})$.
\end{remark}

Suppose $d_1,$ $d_2,$ $\dotsc,$ $d_m$ are positive integers, whose sum is $\sum_{i=1}^md_i=n$. Consider a $4\times 2n$ weight pattern $D_{a,b}:=D_{a,b}(d_1,\dotsc,d_m)$ consisting of
 $m$ blocks $D_i$ of size $4\times 2d_{i}$ from left to right, for $i=1,\dotsc,m$, where blocks $D_i$'s are defined by setting
\[D_{2j-1}:=\begin{bmatrix}
a&a&a&\dots&a&a&a\\
a&a&a&\dots&a&a&a\\
b&b&b&\dots&b&b&b\\
b&b&b&\dots&b&b&b\end{bmatrix}, \text{  and  }
D_{2j}:=\begin{bmatrix}
b&b&b&\dots&b&b&b\\
b&b&b&\dots&b&b&b\\
a&a&a&\dots&a&a&a\\
a&a&a&\dots&a&a&a
\end{bmatrix}.\]

\begin{corollary}\label{fortressw}
Assume that $S$ is defined as in Theorem \ref{main}.

(a) If $n=2k$, then
\begin{equation}\label{ecor}
\M(AD_n(\wt_{D_{a,b}}))=(2ab)^{k(2k+1)-S}(a^2+b^2)^{S}.
\end{equation}

(b) If $n=2k+1$, then
\begin{align}\label{ocor}
\M(AD_n(\wt_{D_{a,b}}))&=\beta\cdot2^{(2k+1)(k+1)-S}a^{2k(k+1)+\theta-S}\notag\\
&\times b^{2k(k+1)+(2k+1)-\theta-S}(a^2+b^2)^{S},
\end{align}
where $\theta=\sum_{odd} d_i$, taken over all odd indices $i\in\{1,2,\dots,m\}$,\\
and where
\begin{equation}
\beta= \begin{cases}
a & \text{if $s_{2j}<k+1\leq s_{2j+1}$ for some $j$;}\\
b&\text{ otherwise.}
\end{cases}
\notag\end{equation}
\end{corollary}

\begin{proof}
This is a special case of Theorem \ref{weighted} when $a_i=b_i=a$ and $c_i=d_i=b$ if $s_{2j}< i\leq s_{2j+1}$, for some $0\leq j \leq (m-1)/2$; and $a_i=b_i=b$ and $c_i=d_i=a$ otherwise.

\medskip
(a) If $n=2k$, then from Theorem \ref{weighted}(a)
\begin{equation}\label{eapply}
\M(AD_n(\wt_{D_{a,b}}))=2^{k(k+1)}\prod_{i=1}^{k}(a^2b^2)^{k-i+1}\prod_{j=1}^{k}\prod_{i=j}^{n-j}\Delta_i,
\end{equation}
because $a_ic_i=b_id_i=ab$, for $i=1,2,\dots,n$. One can check that
\begin{equation}
\Delta_i=
\begin{cases}
a^2+b^2 &\text{if $i=s_t$, for some $1\leq t\leq m$;}\\
2ab &\text{otherwise.}
\end{cases}
\notag\end{equation}
Therefore, we get
\begin{equation}
\prod_{j=1}^{k}\prod_{i=j}^{n-j}\Delta_i=(2ab)^{k^2-\sum_{j=1}^{k}\Psi(j)}(a^2+b^2)^{\sum_{j=1}^{k}\Psi(j)},
\end{equation}
where $\Psi(j)$ is the number of indices $i$ so that $j\leq s_i\leq n-j$. Moreover,
\begin{align}
\sum_{j=1}^{k}\Psi(j)&=\sum_{j=1}^k\sum_{i=1}^{m}\mathbf{1}_{\{j\leq s_{i}\leq n-j\}}(i)=\sum_{i=1}^{m}\sum_{j=1}^k\mathbf{1}_{\{j\leq s_{i}\leq n-j\}}(i)\notag\\
&=\sum_{i=1}^{m}\min(s_{i},n-s_{i})=S,
\end{align}
where ${\textbf{1}}_{\{j\leq s_{i}\leq n-j\}}(i)$ is the indicator function of the event $\{j\leq s_{i}\leq n-j\}$; i.e. it is $1$ if $j\leq s_{i}\leq n-j$, and is $0$ otherwise. This implies (\ref{ecor}).

(b) Assume that $n=2k+1$. We can get (\ref{ocor}) from Theorem \ref{weighted}(b) by arguing similarly to part (a).
  \end{proof}

\medskip

The next two families of graphs will play the key role in investigating the structure of the dual graph of a generalized fortress.

A \textit{regular city} of order $k$ is a row of $k$ adjacent diamonds. An \textit{extended city} of order $k$ is a regular city of order $k$ with two horizontal and $2k$ vertical extended edges (see Figure \ref{Fig9} for examples).

\begin{figure}\centering
\includegraphics[width=0.75\textwidth]{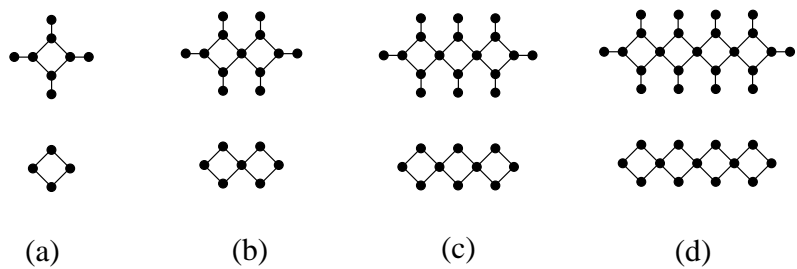}
\caption{Cities of order 1, 2, 3 and 4 (from the left). The extended cities are on the upper row, and normal cities are on the lower row.}
\label{Fig9}
\end{figure}

\begin{figure}\centering
\includegraphics[width=0.75\textwidth]{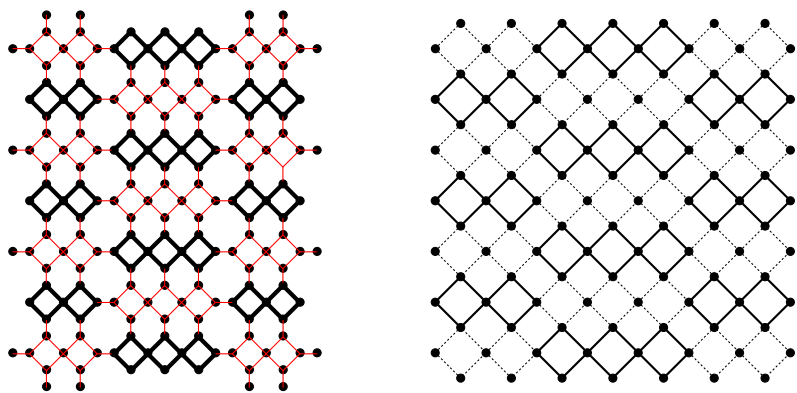}
\caption{The graph $G(2,3,2)$ after rotated $45^o$ (left), and  the graph $AD(2,3,2)$ (right). The dotted edges have weight $1/2$.}
\label{Fig10}
\end{figure}

Denote  by $G(d_1,\dotsc,d_m)$ the dual graph of $F(d_1,\dotsc,d_m)$. The graph $G(d_1,\dotsc,d_m)$ (after rotated $45^o$) consists of $mn$ cities. Precisely, it consists of $n$ rows, and each row has $m$ cities. Regular and extended cities of orders $d_1,d_2,\dotsc, d_m$ appear alternatively on
each row from left to right. All odd rows (ordered from the top)  start with an extended city on the left, and all even rows start with a regular city (illustrated by the left picture in Figure \ref{Fig10}). Denote by $\overline{G}(d_1,\dotsc,d_m)$ the dual graph of the region  $\overline{F}(d_1,\dotsc,d_m)$.
The structures of $\overline{G}(d_1,\dotsc,d_m)$ and $G(d_1,\dots,d_m)$ are (almost) the same, the only difference is that the odd rows now start by a regular city, and the even rows now start by an extended city.

We have the following subgraph replacement similar to urban renewal (see Lemma \ref{spider}(a)).
\begin{figure}\centering
\includegraphics[width=0.75\textwidth]{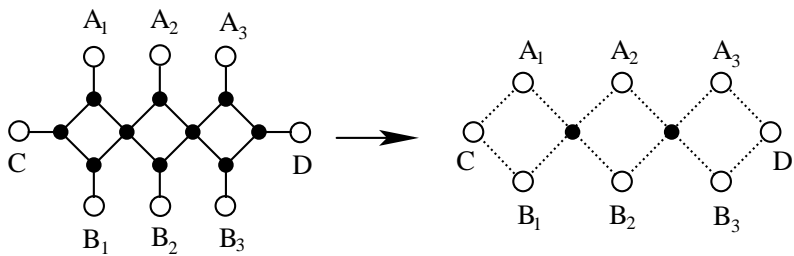}
\caption{The replacement in Lemma \ref{trans} for the cities of order 3. Dotted edges have weight $1/(2x)$.}
\label{Fig11}
\end{figure}

\begin{lemma}\label{trans}
If graph $G$ has a subgraph $K$ isomorphic to an extended city of order $k$ whose edges have weight $x>0$. Assume that only the $2k+2$ endpoints of extended edges of $K$ may have neighbors outside $K$ (illustrated by white vertices of the left graph in Figure \ref{Fig11}, for $k=3$). Let $G'$ be the graph obtained from $G$ by replacing
the extended city $K$ by a regular city of  order $k$ whose edges are weighted by $1/(2x)$ (see the right graph in Figure \ref{Fig11}; $G'$ has $k-1$ new black vertices that were not in $G$). Then $\M(G)=(2x^2)^{k}\M(G')$.
\end{lemma}

\begin{figure}\centering
\includegraphics[width=0.75\textwidth]{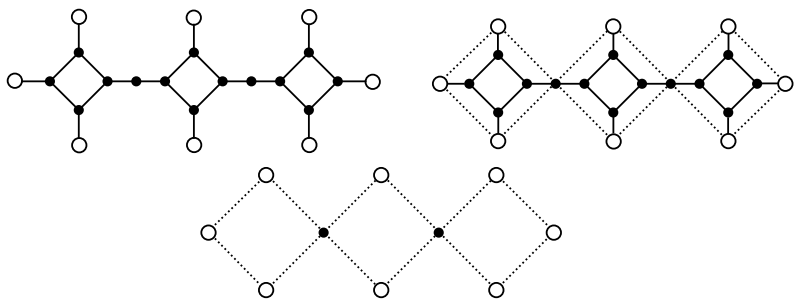}
\caption{Illustrating the proof of Lemma \ref{trans}. The weights of dotted edges are equal to $1/(2x)$.}
\label{Fig12}
\end{figure}

\begin{proof}
   First, apply Vertex-splitting Lemma \ref{VS} at $k-1$ vertices of $K$ that belong to two diamonds, see the left picture in Figure \ref{Fig12}.
Apply Spider Lemma \ref{spider} to $k$  diamond cells in the resulting graph, we get $G'$. By Lemmas \ref{VS} and \ref{spider}, $\M(G)=(2x^2)^{k}\M(G')$.
  \end{proof}

\begin{proof}[Proof of Theorem \ref{main}]
Apply the replacement in Lemma \ref{trans} to each extended city in the graph $G(d_1,\dotsc,d_m)$, we replace each of them by a regular city of the same order whose edges are weighted by 1/2. The resulting graph is isomorphic to the weighted Aztec diamond
\[AD(d_1,\dotsc,d_m):=AD_n(\wt_{D_{\frac{1}{2},1}}),\]
where $D_{\frac{1}{2},1}:=D_{\frac{1}{2},1}(d_1,\dotsc,d_m)$ is defined as in Corollary \ref{fortressw} (see the right picture in Figure \ref{Fig10}), and
\begin{equation}\label{eqC}
\M(G(d_1,\dotsc,d_m))=2^{\mathcal{C}}\,\M(AD(d_1,\dotsc,d_m))
\end{equation}
 where $\mathcal{C}$ is the sum of the sizes of all extended cities in the graph $G(d_1,\dotsc,d_m)$. One readily sees that $\mathcal{C}$ is also the number of  cells of $AD(d_1,\dotsc,d_m)$ whose edges are weighted by $1/2$. Enumerate explicitly these cells we get
\begin{equation}\label{ratio}
\frac{\M(G(d_1,\dotsc,d_m))}{\M(AD(d_1,\dotsc,d_m))}=
\begin{cases}
2^{n^2/2} & \text{if $n$ is even};\\
2^{n(n-1)/2+\sum_{i=0}^{\lfloor(m-1)/2\rfloor}d_{2i+1}} &\text{if $n$ is odd.}
\end{cases}
\end{equation}
Similarly,
\begin{equation}\label{ratio2}
\frac{\M(\overline{G}(d_1,\dots,d_m))}{\M(\overline{AD}(d_1,\dots,d_m))}=
\begin{cases}
2^{n^2/2} & \text{if $n$ is even};\\
2^{n(n-1)/2+\sum_{i=1}^{\lfloor m/2\rfloor}d_{2i}} &\text{if $n$ is odd,}
\end{cases}
\end{equation}
where $\overline{AD}(d_1,\dotsc,d_m):=AD_n(\wt_{D_{1,\frac{1}{2}}})$ and where $D_{1,\frac{1}{2}}:=D_{1,\frac{1}{2}}(d_1,\dotsc,d_m)$ is defined as in the Corollary \ref{fortressw}.

We get (\ref{maineven}) from Corollary \ref{fortressw}(a) (for $a=1/2$ and $b=1$) and  (\ref{ratio}). We also get (\ref{mainodd1}) from  Corollary \ref{fortressw}(b) (for $a=1/2$ and $b=1$) and (\ref{ratio}). Finally, we deduce (\ref{mainodd2})  from  Corollary \ref{fortressw}(b) (for $a=1$ and $b=1/2$) and (\ref{ratio2}).
  \end{proof}

\begin{figure}\centering
\includegraphics[width=0.60\textwidth]{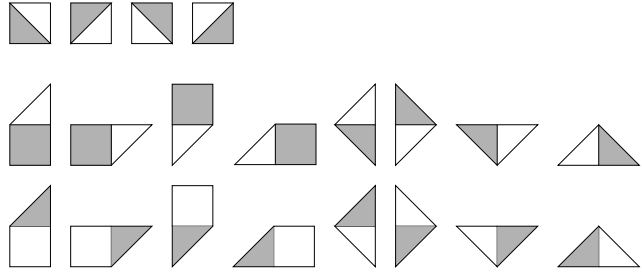}
\caption{All possible types of tiles in a generalized fortress.}
\label{Fig13}
\end{figure}

Similarly to perfect matchings, tilings are allowed to carry weights, in which the weight of a tiling is defined to be the product of the weights on its constituent tiles. The \textit{tiling generating function} of a region $R$, denoted by $\T(R)$, is the sum of the weights of all its tilings. For example, consider the following weight assignment to the tiles of the region $F(d_1,d_2,\dotsc,d_m)$. All  the tiles on the top row of Figure \ref{Fig13} are weighted by 1, all tiles on the middle row are weighted by $1/(2a)$, and all the tiles on the bottom row are weighted by $b$, for some positive numbers $a$ and $b$. Then apply the replacement in Lemma \ref{trans} to all extended cities in the dual graph of the region (an edge in the dual graph has the same weight as its corresponding tile in the region). We get a graph isomorphic to the weighted Aztec diamond $AD_n(\wt_{D_{a.b}})$, where $D_{a,b}$ is defined as in Corollary \ref{fortressw}. Thus,
\begin{equation}
\T(F(d_1,d_2,\dotsc,d_m))=\left(\frac{1}{2a^2}\right)^{\mathcal{C}}\M(AD_n(\wt_{D_{a,b}})),
\notag\end{equation}
where $\mathcal{C}$ is the sum of the sizes of extended cities in the dual graph of $F(d_1,d_2,\dotsc,d_m)$ (as in (\ref{eqC})). It implies that the tiling generating function of $F(d_1,d_2,\dotsc,d_m)$ is a product of several perfect powers.

%%%%%%%%%%%%%%%%%%%%%%%%%%%%%%%%%%%%
\section{The square lattice with zigzag paths}
\label{sec4}

Consider the square lattice with horizontal zigzag paths drawn in (i.e., the bi-infinite paths consist of unit steps going alternatively southeast and northeast), so that the distances between any two consecutive ones are 2 (i.e., we draw them in every second horizontal strip of unit squares). We define next a new family of regions on that lattice as follows.

Let $A$ be a down-pointing vertex of some zigzag path.  Let $\ell$ be the vertical line on the right of $A$ so that the distance from $A$ to $\ell$ is $n\in \mathbb{Z}^+$. Start from $A$ we go periodically by unit steps with the period $\langle\uparrow, \rightarrow,\uparrow, \rightarrow,\nearrow, \rightarrow, \uparrow\rangle$ until reaching the line $\ell$, denote by $B$ the reaching point.  The described path is the northwestern boundary of the region. In the same fashion, we go from $A$  by unit steps with the period  $\langle\rightarrow,\downarrow, \downarrow, \rightarrow,\downarrow,\rightarrow,\searrow\rangle$ until reaching $\ell$, denote by $D$ the new reaching point. The latter path is the southwestern boundary of our region.  The northeastern and southeastern boundaries are obtained by reflecting the northwestern and southwestern boundaries about the line $\ell$, respectively. Denote by $Z_n$ the region bordered by four paths above ($Z_7$ is shown by the  left picture in Figure \ref{Fig14}).

\begin{figure}\centering
\includegraphics[width=0.85\textwidth]{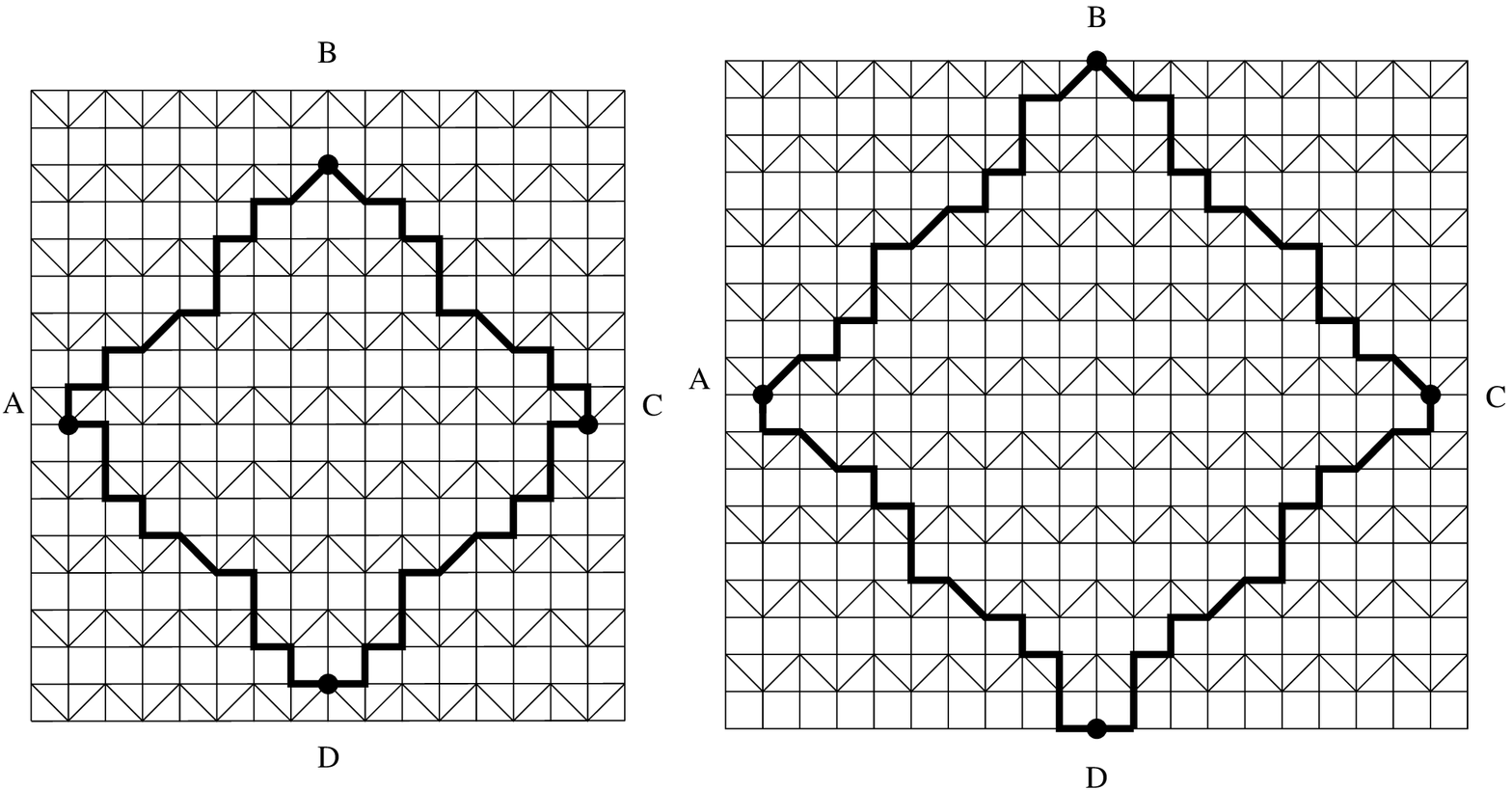}
\caption{Two regions: $Z_7$ (left) and $\overline{Z}_9$ (right).}
\label{Fig14}
\end{figure}

Consider a variant $\overline{Z}_n$ of $Z_n$ defined as follows. We still pick the vertex $A$ and the vertical line $\ell$ as in the definition of the region $Z_n$. The northwestern and southwestern boundaries go from $A$ to $B\in \ell$ and from $A$ to $D\in \ell$ with the periods $\langle\nearrow,\rightarrow,\uparrow,\rightarrow,\uparrow,\uparrow,\rightarrow\rangle$ and $\langle\downarrow,\rightarrow,\searrow,\rightarrow,\downarrow,\rightarrow,\downarrow\rangle$, respectively.  Again, the northeastern and southeastern boundaries are obtained from the previous boundaries by reflecting them about $\ell$. The four described lattice paths complete the boundary of region $\overline{Z}_n$ (see the region on the right in Figure \ref{Fig14} for $\overline{Z}_9$). The numbers of tilings of $Z_n$ and $\overline{Z}_n$ are obtained by the theorem below.

\begin{theorem}\label{ztiling}
(a) For $m\geq 0$
\begin{align}
&\quad\T(Z_{12m})=3^{48m^2}, \quad\quad\quad\quad\quad        \T(Z_{12m+1})=3^{48m^2+8m},\notag\\
&\T(Z_{12m+2})=3^{48m^2+16m+1},\quad\quad\; \T(Z_{12m+3})=3^{48m^2+24m+3},\notag\\
&\T(Z_{12m+4})=2\cdot 3^{48m^2+32m+5},\quad \T(Z_{12m+5})=2\cdot 3^{48m^2+40m+8},\notag\\
&\T(Z_{12m+6})=3^{48m^2+48m+12},\quad\quad
\T(Z_{12m+7})=3^{48m^2+56m+16},\notag\\
&\T(Z_{12m+8})=3^{48m^2+64m+21},\quad\quad
\T(Z_{12m+9})=3^{48m^2+72m+27},\notag\\
&\T(Z_{12m+10})=2\cdot3^{48m^2+80m+33},\;
\T(Z_{12m+11})=2\cdot3^{48m^2+88m+40}.
\notag\end{align}
(b) For nonnegative $m$
\begin{align}
\T(\overline{Z}_{3m})&=\T(Z_{3m}),\notag\\
\T(\overline{Z}_{6m+1})&=2\T(Z_{6m+1}),\notag\\
\T(\overline{Z}_{6m+2})&=2\T(Z_{6m+2}),\notag\\
\T(\overline{Z}_{6m+4})&=\frac{1}{2}\T(Z_{6m+4}),\notag\\
\T(\overline{Z}_{6m+5})&=\frac{1}{2}\T(Z_{6m+5}).\notag
\end{align}
\end{theorem}

We will prove Theorem \ref{ztiling} by using Lemma \ref{ziglem} below. Consider two weight patterns
\begin{equation}B=\begin{bmatrix}
a&a&b&b&b&b&a&a\\
a&a&b&b&b&b&a&a\\
a&a&a&a&b&b&b&b\\
a&a&a&a&b&b&b&b\\
b&b&a&a&a&a&b&b\\
b&b&a&a&a&a&b&b\\
b&b&b&b&a&a&a&a\\
b&b&b&b&a&a&a&a
\end{bmatrix}
\text{ and }
\overline{B}=\begin{bmatrix}
b&b&a&a&a&a&b&b\\
b&b&a&a&a&a&b&b\\
b&b&b&b&a&a&a&a\\
b&b&b&b&a&a&a&a\\
a&a&b&b&b&b&a&a\\
a&a&b&b&b&b&a&a\\
a&a&a&a&b&b&b&b\\
a&a&a&a&b&b&b&b
\end{bmatrix},\notag\end{equation}
where $a$ and $b$ are two positive numbers.

\begin{lemma}\label{ziglem}For $n\geq 3$

(a) \begin{equation}\label{zigrecur1}
\M(AD_{n}(\wt_B))=2^{n}a^{x_n}b^{y_n}(a+b)^{z_n}\M(AD_{n-3}(\wt_{\overline{B}})),
\end{equation}
where  $x_{4k}=8k-1$, $x_{4k+1}=8k+2$, $x_{4k+2}=8k+4$, $x_{4k+3}=8k+5$,  $y_{4k}=8k-2$, $y_{4k+1}=8k-1$, $y_{4k+2}=8k+1$, $y_{4k+3}=8k+4$, and $z_n=2n-3$.

(b) \begin{equation}\label{zigrecur2}
\M(AD_{n}(\wt_{\overline{B}}))=2^{n}a^{\overline{x}_n}b^{\overline{y}_n}(a+b)^{\overline{z}_n}\M(AD_{n-3}(\wt_{B})),
\end{equation}
where $\overline{x}_n=y_n$, $\overline{y}_n=x_n$, and $\overline{z}_n=z_n$.
\end{lemma}

\begin{proof}
The cell-factors of the cells in $AD_{n}(\wt_B)$ are either $2a^2$ or $2b^2$, thus by Reduction Theorem
\begin{equation}\label{zig1}
\M(AD_{n}(\wt_B))=2^{n^2}a^{k_1}b^{k_2}\M(AD_{n-1}(\wt_{d(B)})),
\end{equation}
where $k_1,k_2$ are integers given in (\ref{k1}) and (\ref{k2}),  and where
\[d(B)=\begin{bmatrix}
1/2a&1/2b&1/2b&1/2b&1/2b&1/2a&1/2a&1/2a\\
1/2a&1/2a&1/2a&1/2b&1/2b&1/b&1/2b&1/2a\\
1/2a&1/2a&1/2a&1/2b&1/2b&1/b&1/2b&1/2a\\
1/2b&1/2a&1/2a&1/2a&1/2a&1/2b&1/2b&1/2b\\
1/2b&1/2a&1/2a&1/2a&1/2a&1/2b&1/2b&1/2b\\
1/2b&1/2b&1/2b&1/2a&1/2a&1/2a&1/2a&1/2b\\
1/2b&1/2b&1/2b&1/2a&1/2a&1/2a&1/2a&1/2b\\
1/2a&1/2b&1/2b&1/2b&1/2b&1/2a&1/2a&1/2a
\end{bmatrix}.\]
Since the cell-factors of the cells in $AD_{n-1}(\wt_{d(B)})$ are either $\dfrac{a+b}{4a^2b}$ or $\dfrac{a+b}{4ab^2}$,  the Reduction Theorem yields
\begin{equation}\label{zig2}
\M(AD_{n-1}(\wt_{\de(B)}))=\left(\dfrac{a+b}{4a^2b}\right)^{k_3}\left(\dfrac{a+b}{4ab^2}\right)^{k_4}\M(AD_{n-2}(\wt_{\de^{(2)}(B)})),
\end{equation}
where $k_3$ and $k_4$ are integers given in (\ref{k3}) and (\ref{k4}), and where
\[\de^{(2)}(A)=\begin{bmatrix}
\frac{2 a b}{b+a}&\frac{2 a b}{b+a}&\frac{2 a b}{b+a}&\frac{2 b^2}{b+a}&\frac{2 a b}{b+a}&\frac{2 a b}{b+a}&\frac{2 a b}{b+a}&\frac{2 a^2}{b+a}\\
\frac{2 a^2}{b+a}&\frac{2 a b}{b+a}&\frac{2 a b}{b+a}&\frac{2 a b}{b+a}&\frac{2 b^2}{b+a}&\frac{2 a b}{b+a}&\frac{2 a b}{b+a}&\frac{2 a b}{b+a}\\
\frac{2 a b}{b+a}&\frac{2 a^2}{b+a}&\frac{2 a b}{b+a}&\frac{2 a b}{b+a}&\frac{2 a b}{b+a}&\frac{2 b^2}{b+a}&\frac{2 a b}{b+a}&\frac{2 a b}{b+a}\\
\frac{2 a b}{b+a}&\frac{2 a b}{b+a}&\frac{2 a^2}{b+a}&\frac{2 a b}{b+a}&\frac{2 a b}{b+a}&\frac{2 a b}{b+a}&\frac{2 b^2}{b+a}&\frac{2 a b}{b+a}\\
\frac{2 a b}{b+a}&\frac{2 a b}{b+a}&\frac{2 a b}{b+a}&\frac{2 a^2}{b+a}&\frac{2 a b}{b+a}&\frac{2 a b}{b+a}&\frac{2 a b}{b+a}&\frac{2 b^2}{b+a}\\
\frac{2 b^2}{b+a}&\frac{2 a b}{b+a}&\frac{2 a b}{b+a}&\frac{2 a b}{b+a}&\frac{2 a^2}{b+a}&\frac{2 a b}{b+a}&\frac{2 a b}{b+a}&\frac{2 a b}{b+a}\\
\frac{2 a b}{b+a}&\frac{2 b^2}{b+a}&\frac{2 a b}{b+a}&\frac{2 a b}{b+a}&\frac{2 a b}{b+a}&\frac{2 a^2}{b+a}&\frac{2 a b}{b+a}&\frac{2 a b}{b+a}\\
\frac{2 a b}{b+a}&\frac{2 a b}{b+a}&\frac{2 b^2}{b+a}&\frac{2 a b}{b+a}&\frac{2 a b}{b+a}&\frac{2 a b}{b+a}&\frac{2 a^2}{b+a}&\frac{2 a b}{b+a}
\end{bmatrix}.\]
One readily sees that a cell in $AD_{n-2}(\wt_{\de^{(2)}(A)})$ has cell-factor either $\dfrac{4a^2b}{a+b}$ or $\dfrac{4ab^2}{a+b}$. Apply the Reduction Theorem again, we obtain
\begin{equation}\label{zig3}
\M(AD_{n-2}(\wt_{\de^{(2)}(B)}))=\left(\dfrac{4a^2b}{a+b}\right)^{k_5}\left(\dfrac{4ab^2}{a+b}\right)^{k_6}\M(AD_{n-3}(\wt_{\de^{(3)}(B)})),
\end{equation}
where $k_5$ and $k_6$ are integers given in (\ref{k5}) and (\ref{k6}), and where $\de^{(3)}(B)=\frac{1}{2ab}\overline{B}$. Thus, we get from Lemma \ref{mulstar}
\begin{equation}\label{zig4}
\M(AD_{n-3}(\wt_{\de^{(3)}(B)}))=\left(\frac{1}{2ab}\right)^{(n-2)(n-3)}\M(AD_{n-3}(\wt_{\overline{B}})).
\end{equation}

We get by calculating explicitly
\begin{equation}\label{k1}k_1=\begin{cases}
8k^2&\text{ if $n=4k$;}\\
8k^2+4k+1&\text{ if $n=4k+1$;}\\
8k^2+8k+3&\text{ if $n=4k+2$;}\\
8k^2+8k+5&\text{ if $n=4k+3$,}
\end{cases}
\end{equation}
\begin{equation}\label{k2}k_2=\begin{cases}
8k^2&\text{ if $n=4k$;}\\
8k^2+4k&\text{ if $n=4k+1$;}\\
8k^2+8k+1&\text{ if $n=4k+2$;}\\
8k^2+8k+4&\text{ if $n=4k+3$,}
\end{cases}
\end{equation}
\begin{equation}\label{k3}k_3=\begin{cases}
8(k-1)^2+12(k-1)+5&\text{ if $n=4k$;}\\
8k^2&\text{ if $n=4k+1$;}\\
8k^2+4k+1&\text{ if $n=4k+2$;}\\
8k^2+8k+3&\text{ if $n=4k+3$,}
\end{cases}
\end{equation}
\begin{equation}\label{k4}k_4=\begin{cases}
8(k-1)^2+12(k-1)+4&\text{ if $n=4k$;}\\
8k^2&\text{ if $n=4k+1$;}\\
8k^2+4k&\text{ if $n=4k+2$;}\\
8k^2+8k+1&\text{ if $n=4k+3$,}
\end{cases}
\end{equation}
\begin{equation}\label{k5}k_5=\begin{cases}
8(k-1)^2+8(k-1)+3&\text{ if $n=4k$;}\\
8(k-1)^2+12(k-1)+5&\text{ if $n=4k+1$;}\\
8k^2&\text{ if $n=4k+2$;}\\
8k^2+4k+1&\text{ if $n=4k+3$,}\\
\end{cases}
\end{equation}
and
\begin{equation}\label{k6}
k_6=\begin{cases}
8(k-1)^2+8(k-1)+1&\text{ if $n=4k$;}\\
8(k-1)^2+12(k-1)+4&\text{ if $n=4k+1$;}\\
8k^2&\text{ if $n=4k+2$;}\\
8k^2+4k&\text{ if $n=4k+3$}.\\
\end{cases}
\end{equation}
By (\ref{zig1})--(\ref{k6}), we obtain the recurrence in part (a). Part (b) is absolutely analogous.
  \end{proof}

\begin{proof}[Proof of Theorem \ref{ztiling}]
Apply the Vertex-splitting Lemma \ref{VS} to all circled vertices in the dual graph of $Z_n$ as in Figure \ref{Fig15}, for $n=7$ (the general case can be treated similarly). Apply suitable replacements in Spider Lemma \ref{spider} to all shaded cells  and shaded partial cells in the resulting graph (illustrated by  Figure \ref{Fig16}; the dotted edges have weight $1/2$, and all shaded cells and partial cells inside the dotted contours will be removed). We get a graph isomorphic to the weighted Aztec diamond of order $n$ with weight pattern

\begin{figure}\centering
\includegraphics[width=0.75\textwidth]{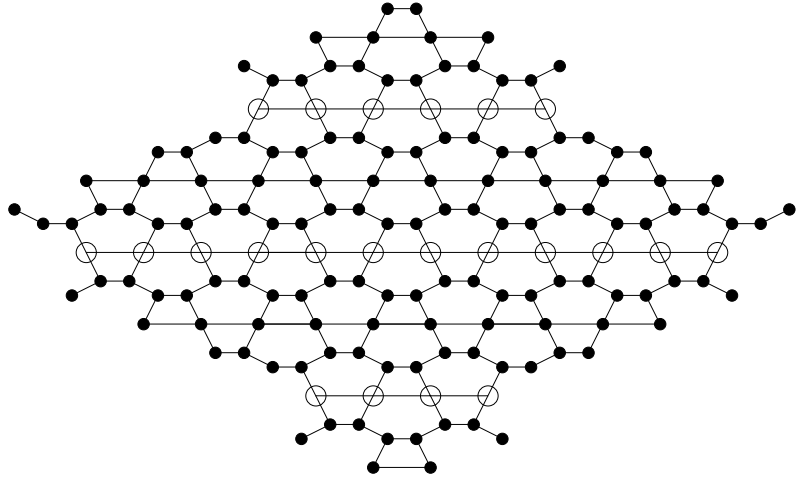}
\caption{The dual graph of $Z_7$.}
\label{Fig15}
\end{figure}

\begin{figure}\centering
\includegraphics[width=0.75\textwidth]{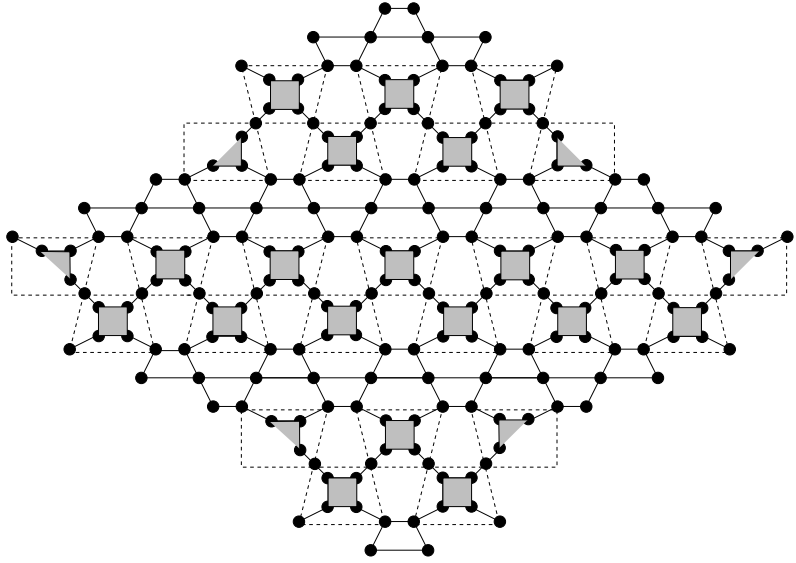}
\caption{The graph obtained from the dual graph of $Z_7$ by applying Vertex-splitting Lemma.}
\label{Fig16}
\end{figure}

\[A=\begin{bmatrix}
1/2&1/2&1&1&1&1&1/2&1/2\\
1/2&1/2&1&1&1&1&1/2&1/2\\
1/2&1/2&1/2&1/2&1&1&1&1\\
1/2&1/2&1/2&1/2&1&1&1&1\\
1&1&1/2&1/2&1/2&1/2&1&1\\
1&1&1/2&1/2&1/2&1/2&1&1\\
1&1&1&1&1/2&1/2&1/2&1/2\\
1&1&1&1&1/2&1/2&1/2&1/2\\
\end{bmatrix}.\]
By Lemmas \ref{VS} and \ref{spider}, we deduce that
\begin{equation}\label{gama1}
\T(Z_n)=2^{\gamma_n} \M(AD_n(\wt_A)),
\end{equation}
where $\gamma_{4k}=8k^2$, $\gamma_{4k+1}=8k^2+4k+1$,  $\gamma_{4k+2}=8k^2+8k+3$,  and $\gamma_{4k+3}=8k^2+12k+5$, for $k\geq 0$.

\begin{figure}\centering
\includegraphics[width=0.75\textwidth]{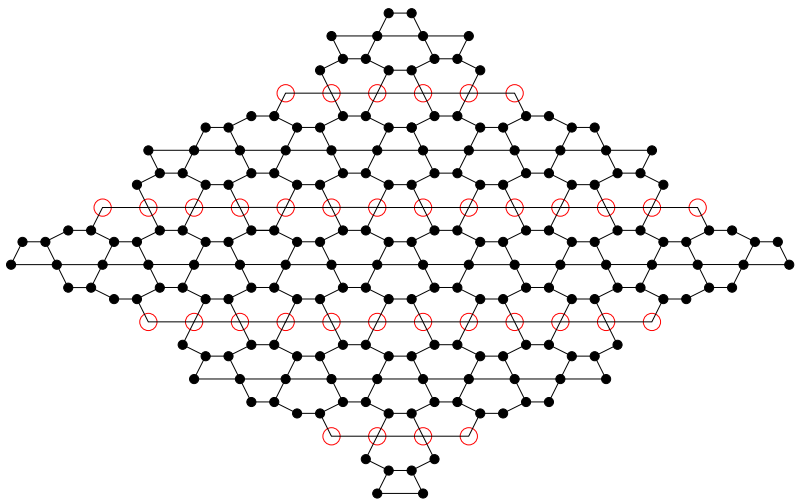}
\caption{The dual graph of $\overline{Z}_9$.}
\label{Fig17}
\end{figure}

\begin{figure}\centering
\includegraphics[width=0.75\textwidth]{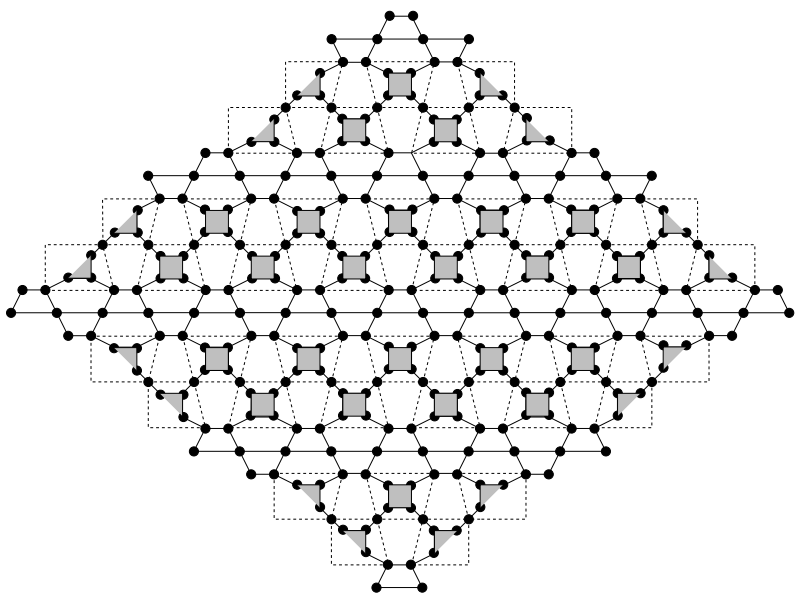}
\caption{The graph obtained from the dual graph of $\overline{Z}_9$ by applying Vertex-splitting Lemma.}
\label{Fig18}
\end{figure}

Do similarly for the dual graph of $\overline{Z}_n$ (shown in Figures \ref{Fig17} and \ref{Fig18}, for $n=9$). The resulting graph is isomorphic to the weighted Aztec diamond of order $n$ with the weight pattern
\[\overline{A}=\begin{bmatrix}
1&1&1/2&1/2&1/2&1/2&1&1\\
1&1&1/2&1/2&1/2&1/2&1&1\\
1&1&1&1&1/2&1/2&1/2&1/2\\
1&1&1&1&1/2&1/2&1/2&1/2\\
1/2&1/2&1&1&1&1&1/2&1/2\\
1/2&1/2&1&1&1&1&1/2&1/2\\
1/2&1/2&1/2&1/2&1&1&1&1\\
1/2&1/2&1/2&1/2&1&1&1&1
\end{bmatrix},\]
and
\begin{equation}\label{gama2}
\T(\overline{Z}_n)=2^{\overline{\gamma}_n}\M(AD_n(\wt_{\overline{A}})),
\end{equation}
 where $\overline{\gamma}_{4k}=8k^2$, $\overline{\gamma}_{4k+1}=8k^2+4k$,  $\overline{\gamma}_{4k+2}=8k^2+8k+1$,  and $\overline{\gamma}_{4k+3}=8k^2+12k+4$, for $k\geq0$.

By Reduction Theorem, we get that the three initial values of $\M(AD_{n}(\wt_A))$, and that the three initial values of $\M(AD_{n}(\wt_{\overline{A}}))$ are

$\M(AD_0(\wt_A))=1$,

$\M(AD_1(\wt_A))=2^{-1},$

 $\M(AD_2(\wt_A))=2^{-3}\cdot3,$

 $\M(AD_0(\wt_{\overline{A}}))=1$,

 $\M(AD_1(\wt_{\overline{A}}))=2,$

 $\M(AD_2(\wt_{\overline{A}}))=3.$\\
For $n\geq3$ the values of $\M(AD_{n}(\wt_A))$ and $\M(AD_{n}(\wt_{\overline{A}}))$ are obtained from Lemma \ref{ziglem} by specializing $a=1/2$ and $b=1$. Thus, the theorem follows from (\ref{gama1}) and (\ref{gama2}).
  \end{proof}

In previous work, Blum has considered a different family of subgraphs of the square grid for which he noticed that the number of perfect matching seems to be always a power of 3 or twice a power of 3.  Consider the sublattice of the square lattice showed in Figure \ref{Fig19} and view it as an infinite graph $G$.  In particular, each row in this lattice consists of $1\times 2$ and $1\times 3$ bricks, that occur alternatively. All the odd rows are the same, and the even rows are obtained from the left rows by shifting  them  one unit to the right. Draw a boundary of the Aztec diamond graph of order $n$ on this lattice so that the easternmost edge has an embedded hexagon east of it. Let $B_n$ be the induced subgraph of $G$ spanned by the vertices lying inside or on the boundary of the Aztec diamond graph.  Ciucu proved the Blum's conjecture by applying the Reduction Theorem \textit{30 times} (!). In particular, we have the following theorem.

\begin{figure}\centering
\includegraphics[width=0.75\textwidth]{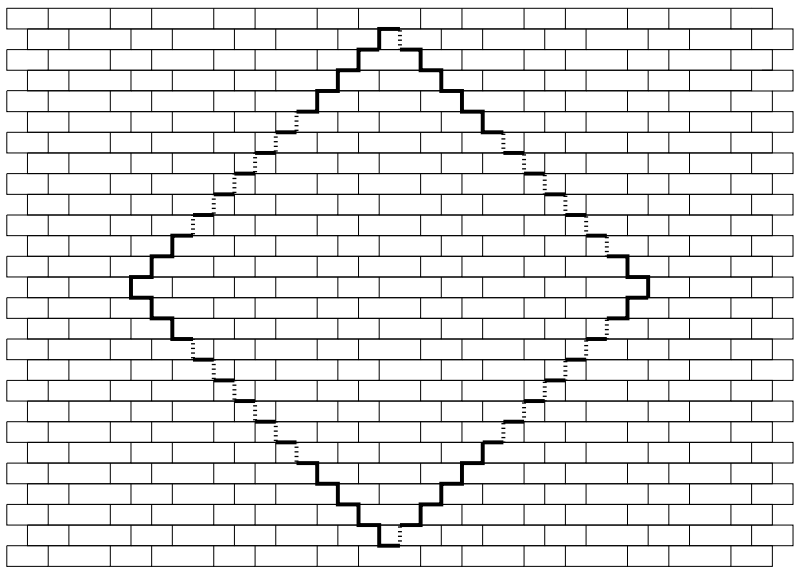}
\caption{The graph $B_{13}$.}
\label{Fig19}
\end{figure}

\begin{theorem}[Ciucu \cite{Ciucu5}]\label{Ciucu-brick}
For $n\geq 31$ we have
\begin{equation}
\M(B_n)=3^{4x_n}\M(B_{n-30}),
\notag\end{equation}
where $x_{5k+1}=4k-12$, $x_{5k+2}=4k-10$ and $x_{5k+3}=x_{5k+4}=x_{5k+5}=4k-8$, $k\geq 6$.
\end{theorem}

One can realize that the number of tilings of $Z_n$ ($\overline{Z}_n$) and the number of perfect matchings of $B_n$ are similar. They are both a power of $3$ or twice a power of $3$. Is there any relationship between these values? We have the answer for the latter question in the next part of this section. Moreover, the answer provides a new proof for the conjecture of Blum above.

 Consider a new lattice that is obtained from the lattice in the Blum's conjecture by replacing all $1\times 3$ bricks by  $1\times1$  bricks. Draw also the boundary of the Aztec diamond of order $n$ on the new lattice, so that the easternmost edge has an embedded hexagon east of it. Again, let $C_n$ be the induced subgraph spanned by the vertices lying inside or on the boundary of the Aztec diamond (see the right picture of Figure \ref{Fig20}).

\begin{figure}\centering
\includegraphics[width=0.90\textwidth]{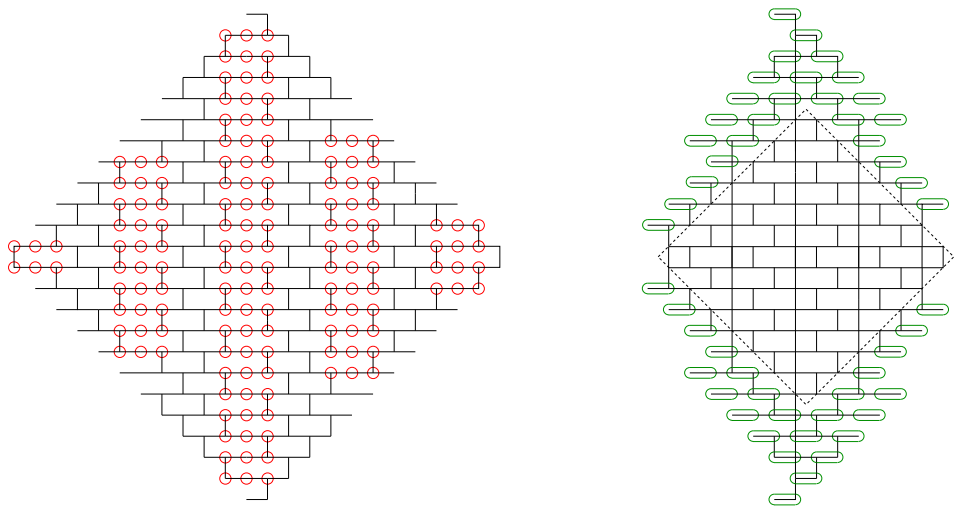}
\caption{Illustrating the transformation the graph $B_{12}$ (left) into the graph $C_7$ (inside the dotted diamond contour on the right).}
\label{Fig20}
\end{figure}

Apply Vertex-splitting Lemma (in reverse) to identify all three consecutive circled vertices in a row of  $B_{5k+2}$ as in the left picture of Figure \ref{Fig20}, for $k=2$ (the general case can be obtained similarly).  Deform the resulting graph into a subgraph of the new lattice above. After removing some horizontal forced edges we get the graph $C_{3k+1}$ (illustrated by the right picture in Figure \ref{Fig20}; the forced edges are circled).
Thus,
\begin{equation}\label{z1}
\M(B_{5k+2})=\M(C_{3k+1}).
\end{equation}
Apply the same process to graph $B_{5k-2}$ we get graph $C_{3k-1}$, so
\begin{equation}\label{z1'}
\M(B_{5k-2})=\M(C_{3k-1}).
\end{equation}

\begin{figure}\centering
\includegraphics[width=0.75\textwidth]{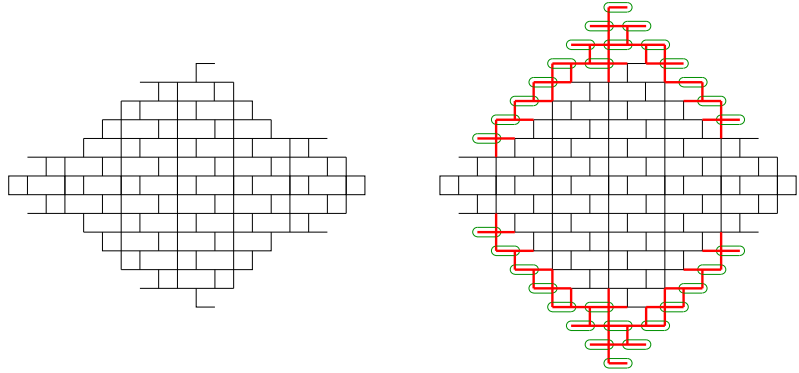}
\caption{The deformed version of the dual graph of $Z_7$ (left), and the graph $C_{10}$ (right).}
\label{Fig21}
\end{figure}

Rotate $45^0$ the dual graph of $Z_n$, and deform the resulting graph into a subgraph of the new lattice. One can see that the dual graph of  $Z_{4k+3}$ is obtained from $C_{6k+4}$ by removing some horizontal forced edges (see Figure \ref{Fig21} for an example).
 Therefore, we imply that
\begin{equation}\label{z2}
\T(Z_{4k+3})=\M(C_{6k+4}).
\end{equation}
In the same fashion, we have
\begin{equation}\label{z3}
\T(Z_{4k})=\M(C_{6k-1}),\\
\T(\overline{Z}_{4k+1})=\M(C_{6k+1}), \text{ and }  \T(\overline{Z}_{4k+2})=\M(C_{6k+2}).
\end{equation}
We get the following equalities by applying (\ref{z1})--(\ref{z3}):
\begin{equation}\label{zigrelate1}
\T(Z_{4k+3})=\M(B_{10k+7}),
\end{equation}
\begin{equation}\label{zigrelate2}
\T(Z_{4k})=\M(B_{10k-2}),
\end{equation}
\begin{equation}\label{zigrelate3}
\T(\overline{Z}_{4k+1})=\M(B_{10k+2}),
\end{equation}
\begin{equation}\label{zigrelate4}
\T(\overline{Z}_{4k+2})=\M(B_{10k+3}).
\end{equation}
Moreover, by considering horizontal forced edges in graph $B_n$, we can verify easily that
\begin{equation}\label{zigrelate5}
\M(B_{5k-2})=\M(B_{5k-1})=\M(B_{5k})=\M(B_{5k+1}),
\end{equation}
 for $k\geq 1$. Thus,  Theorem \ref{ztiling}  and equalities (\ref{zigrelate1})--(\ref{zigrelate5}) yield a new proof for the Blum's ex-conjecture.

\section{Variants of the square lattice}
\label{sec5}

The square lattice has the set of nodes $\mathbb{Z}^2=\{(x,y)|x,y\in \mathbb{Z}\}$. The two subgraph replacement rules in Figure \ref{Fig22} will play the key roles in this section.  In particular, the local subgraph around a node, i.e. the cross on the left of Figures \ref{Fig22}(a) and (b), is replaced by the one on the right with corresponding black vertices and two new white vertices.   The graph replacement in Figure \ref{Fig22}(a) (resp. Figure \ref{Fig22}(b)) is called the first (resp. the second) \textit{node replacement}.
 We will get a large number of variants of the square lattice by  periodically applying these node replacements (together with some simple modifications). We will go over several examples in the next part of this section. In those examples, certain families of regions have the numbers of tilings given by products of perfect powers.

\begin{figure}\centering
\includegraphics[width=0.9\textwidth]{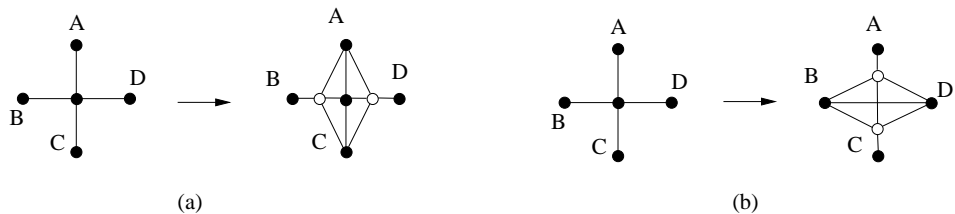}
\caption{Two node replacements.}
\label{Fig22}
\end{figure}

Denote by $\mathbb{D}_m(a,b)$ the diamond of side-length $m\sqrt{2}$ whose western vertex is the node $(a,b)$, where $a$ and $b$ are some integers. Apply the first node replacement to  all nodes $(4k,4l\pm 1)$, and apply the second node replacement to all nodes $(4k+2, 4l\pm1)$, for  any integers $k$ and $l$. Consider a region on the resulting lattice that consists of all elementary regions lying completely or partially inside the diamond $\mathbb{D}_m(-1,1)$ (see the left picture in Figure \ref{Fig23} for an example). Denote by $S^{(1)}_m$ this region. The number of tilings of $S^{(1)}_m$ is given by the following theorem.

\begin{figure}\centering
\includegraphics[width=0.90\textwidth]{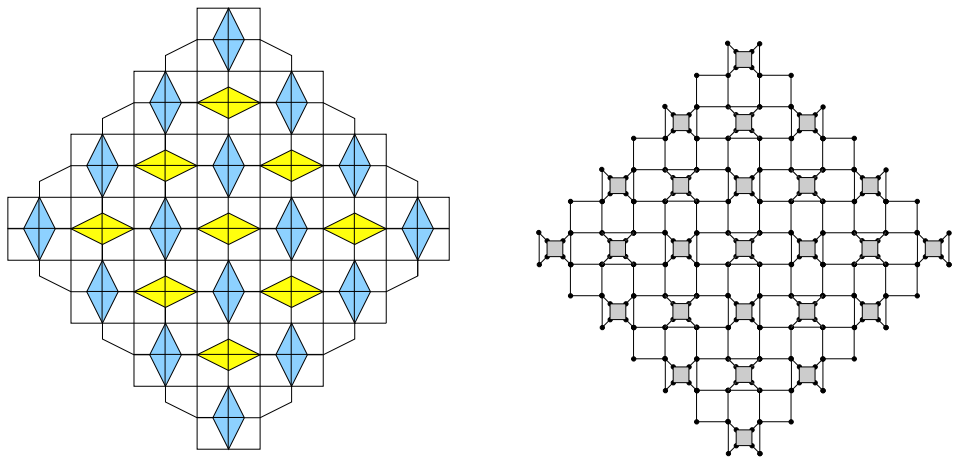}
\caption{The region $S^{(1)}_7$ (left) and its dual graph (right).}
\label{Fig23}
\end{figure}

\begin{theorem}\label{power7thm} For any $n\geq 0$
\[\T(S^{(1)}_{4n})=7^{n(2n-1)}37^{n(2n+1)}, \ \ \  \T(S^{(1)}_{4n+1})=5\cdot7^{2n^2}37^{2n(n+1)},\]
\[\T(S^{(1)}_{4n+2})=7^{n(2n+1)}37^{(n+1)(2n+1)},\ \ \  \T(S^{(1)}_{4n+3})=2\cdot7^{2n(n+1)}37^{2n(n+2)+2}.\]
\end{theorem}

Before presenting the proof of Theorem \ref{power7thm}, we consider the following weight pattern
\begin{equation}\label{patternB}
B=\begin{bmatrix}
a&b&c&d\\
b&a&d&c\\
d&c&b&a\\
c&d&a&b
\end{bmatrix},
\end{equation}
where $a,b,c$ and $d$ are four positive numbers. The matching generating function of $AD_n(\wt_B)$ is obtained by the theorem below.
\begin{theorem}\label{7gen}  For any nonnegative integer $k$
\begin{align}
\M(AD_{4k}(\wt_B))&=P^{k(2k+1)}(ab+cd)^{k(2k-1)},\notag\\
\M(AD_{4k+1}(\wt_B))&=(a^2+b^2)\, P^{k(2k+2)}(ab+cd)^{2k^2},\notag\\
\M(AD_{4k+2}(\wt_B))&=P^{k(2k+3)+1}(ab+cd)^{k(2k+1)},\notag\\
\M(AD_{4k+3}(\wt_B))&=P^{k(2k+4)+2}(ab+cd)^{k(2k+2)},\notag
\end{align}
where $P=ab(c^2+d^2)^2+cd(a^2+b^2)^2$.
\end{theorem}

\begin{proof}
Assume that $m=2n$, for some positive integer $n$. The cell-factors (of cells) in $AD_{2n}(\wt_B)$ are either $(a^2+b^2)$ or $(c^2+d^2)$, so Reduction Theorem implies
\begin{equation}\label{7gen1}
\M(AD_{2n}(\wt_B))=(a^2+b^2)^{2n^2}(c^2+d^2)^{2n^2}\M(AD_{2n-1}(\wt_{d(B)})),
\end{equation}
where
\[\de(B)=\begin{bmatrix}
a/(a^2+b^2)&d/(c^2+d^2)&c/(c^2+d^2)&b/(b^2+a^2)\\
c/(c^2+d^2)&b/(a^2+b^2)&a/(a^2+b^2)&d/(c^2+d^2)\\
d/(c^2+d^2)&a/(a^2+b^2)&b/(a^2+b^2)&c/(c^2+d^2)\\
b/(a^2+b^2)&c/(c^2+d^2)&d/(c^2+d^2)&a/(a^2+b^2)
\end{bmatrix}.\]
Since the cell-factors in $AD_{2n-1}(\wt_{d(B)})$ are all $\Delta_1:=\left(ab/(a^2+b^2)^2+cd/(c^2+d^2)^2\right)$, we apply the Reduction Theorem again and obtain
\begin{equation}\label{7gen2}
\M(AD_{2n-1}(\wt_{\de(B)}))=\Delta_1^{(2n-1)^2}\M(AD_{2n-2}(\wt_{\de^{(2)}(B)})),
\end{equation}
where $\de^{(2)}(B)=\frac{1}{\Delta_1}\,C$ with
\[C=\begin{bmatrix}
a/(a^2+b^2)&b/(a^2+b^2)&c/(c^2+d^2)&d/(c^2+d^2)\\
b/(a^2+b^2)&a/(a^2+b^2)&d/(c^2+d^2)&c/(c^2+d^2)\\
d/(c^2+d^2)&c/(c^2+d^2)&b/(a^2+b^2)&a/(a^2+b^2)\\
c/(c^2+d^2)&d/(c^2+d^2)&a/(a^2+b^2)&b/(a^2+b^2)
\end{bmatrix}.\]
Therefore, the following equality follows from Lemma \ref{mulstar}
\begin{equation}\label{7gen3}
\M(AD_{2n-2}(\wt_{\de^{(2)}(B)}))=\Delta_1^{-(2n-1)(2n-2)}\M(AD_{2n-2}(\wt_{C})).
\end{equation}

We define an operator $r$ on the space of $4\times 4$ matrices by setting
\begin{equation}
r\left(\begin{bmatrix}
a&b&c&d\\
b&a&d&c\\
d&c&b&a\\
c&d&a&b
\end{bmatrix}\right):=\begin{bmatrix}
a/(a^2+b^2)&b/(a^2+b^2)&c/(c^2+d^2)&d/(c^2+d^2)\\
b/(a^2+b^2)&a/(a^2+b^2)&d/(c^2+d^2)&c/(c^2+d^2)\\
d/(c^2+d^2)&c/(c^2+d^2)&b/(a^2+b^2)&a/(a^2+b^2)\\
c/(c^2+d^2)&d/(c^2+d^2)&a/(a^2+b^2)&b/(a^2+b^2)
\end{bmatrix}.
\end{equation}
One readily sees  $r(B)=C$ and  $r^{(2)}(B)=B$. By (\ref{7gen1}), (\ref{7gen2}) and (\ref{7gen3})
 \begin{equation}\label{7gen4}
 \M(AD_{2n}(\wt_B))=(a^2+b^2)^{2n^2}(c^2+d^2)^{2n^2}\Delta_1^{2n-1}\M(AD_{2n-2}(\wt_{r(B)})).
 \end{equation}
We apply the recurrence  (\ref{7gen4}) to $AD_{2n-2}(\wt_{r(B)})$ and obtain
\begin{equation}\label{7gen5}
\M(AD_{2n-2}(\wt_{r(B)}))=(a^2+b^2)^{-2(n-1)^2}(c^2+d^2)^{-2(n-1)^2}\Delta_2^{2n-3}\M(AD_{2n-4}(\wt_{B})),
\end{equation}
where $\Delta_2:=ab+cd$.  Two equalities (\ref{7gen4}) and (\ref{7gen5}) imply
\begin{align}\label{power7eq1}
\frac{\M(AD_{2n}(\wt_B))}{\M(AD_{2n-4}(\wt_B))}&=(a^2+b^2)^{2(2n-1)}(c^2+d^2)^{2(2n-1)}\Delta_1^{2n-1}(ab+cd)^{2n-3}\notag\\
&=P^{2n-1}(ab+cd)^{2n-3}.
\end{align}
Do similarly for the case $m=2n+1$, we get
\begin{equation}\label{power7eq2}
\frac{\M(AD_{2n+1}(\wt_B))}{\M(AD_{2n-3}(\wt_B))}=P^{2n}(ab+cd)^{2n-2}.
\end{equation}
The theorem follows from the two recurrences (\ref{power7eq1}) and (\ref{power7eq2}).
  \end{proof}
\begin{proof}[Proof of Theorem \ref{power7thm}]
 Apply the replacement in Spider Lemma \ref{spider}(a) to all  shaded cells in the dual graph of $S^{(1)}_m$ (see the right picture in Figure \ref{Fig23}). Then apply Edge-replacing Lemma \ref{ER} to all multiple edges arising from the previous step. This process gives us a graph isomorphic to the weighted Aztec diamond of order $m$ with weight pattern
\[B_1=\begin{bmatrix}
3/2&1/2&1&1\\
1/2&3/2&1&1\\
1&1&1/2&3/2\\
1&1&3/2&1/2
\end{bmatrix}.
\]

One readily sees that the number of shaded cells is $2n^2$ if $m=2n$, and is $(n+1)^2+n^2$ if $m=2n+1$. Thus, by Lemmas \ref{ER} and \ref{spider}, we get
\begin{equation}\label{p7eq1}
\frac{\T(S^{(1)}_m)}{\M(AD_{m}(\wt_{B_1}))}=\begin{cases}
2^{2n^2} &\text{if $m=2n$;}\\
2^{(n+1)^2+n^2}&\text{if $m=2n+1$.}\end{cases}
\end{equation}
The theorem follows from Theorem \ref{7gen} (for $a=3/2$, $b=1/2$ and $c=d=1$) and (\ref{p7eq1}).
  \end{proof}

\begin{figure}\centering
\includegraphics[width=0.65\textwidth]{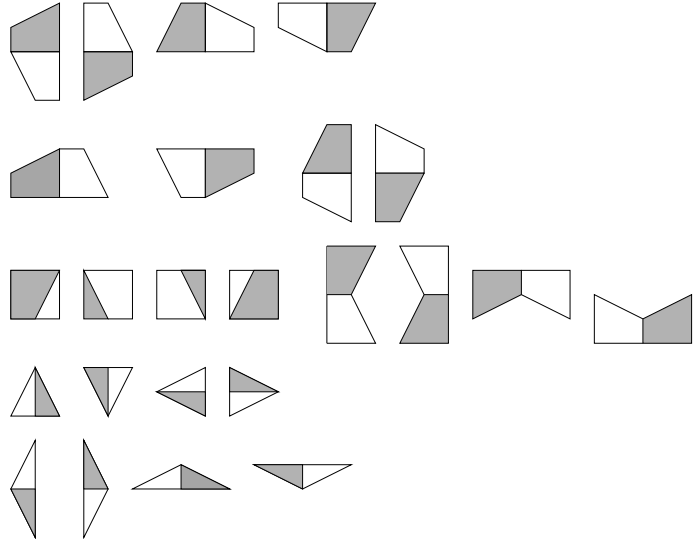}
\caption{All possible types of tiles in $S^{(1)}_m$.}
\label{Fig24}
\end{figure}

Similarly to what we did for the generalized fortress $F(d_1,d_2,\dotsc,d_m)$, we can assign weights to the tiles of the region $S^{(1)}_m$ as follows.  All tiles on the first row in Figure \ref{Fig24} are weighted by $c$, all tiles on the second row are weighted by $d$, the tiles on the third row have weight 1, the tiles on the fourth row have weight $\frac{b}{(a-1)^2+b^2}$, finally all tiles on the last row are weighted by $\frac{a-1}{(a-1)^2+b^2}$, where $a,b,c,d>0$. After applying Spider Lemma to all shaded cells as in the proof of Theorem \ref{power7thm}, we get a graph isomorphic to $AD_m(\wt_B)$, where $B$ is defined by (\ref{patternB}). Thus, we have the following equality instead of (\ref{p7eq1})
\begin{equation}\label{p7eq1'}
\frac{\T(S^{(1)}_m)}{\M(AD_{m}(\wt_{B}))}=\begin{cases}
\left(\frac{1}{(a-1)^2+b^2}\right)^{2n^2} &\text{if $m=2n$;}\\
\left(\frac{1}{(a-1)^2+b^2}\right)^{(n+1)^2+n^2}&\text{if $m=2n+1$.}\end{cases}
\end{equation}
By Theorem \ref{7gen} and (\ref{p7eq1'}), the tiling generating function of $S^{(1)}_m$ (with the new weight assignment to its tiles) is given by a product of several perfect powers.

\medskip
Start with the square lattice with every second diagonal drawn in (see the right picture in Figure \ref{Fig6}). Apply the first node replacement to all nodes $(4k+3,4l)$ and $(4k+1,4l+2)$, and apply the the second node replacement to all nodes $(4k+1,4l)$ and $(4k+3,4l+2)$,  for any integers $k$ and $l$. The elementary regions in the resulting lattice are all triangles. Consider two families of the variants of the fortresses on this lattice as in  Figure \ref{Fig25}.  In particular, the region of order $n$ in the first family consists of all elementary regions lying completely inside the diamond $\mathbb{D}_n(0,0)$, together with all elementary regions that are obtuse triangles  having an edge on the boundary of $\mathbb{D}_n(0,0)$. This region is denoted by $S^{(2)}_n$ (see the left picture in Figure \ref{Fig25}). The region of order $n$ in the second family consists of all elementary regions lying completely inside $\mathbb{D}_n(0,0)$, together with all elementary regions that are right isosceles triangles having hypothenuse on the boundary of $\mathbb{D}_n(0,0)$. Denote by $S^{(3)}_n$ the latter region (the region $S^{(3)}_6$ is illustrated by the right picture in Figure \ref{Fig25}).

\begin{figure}\centering
\includegraphics[width=0.90\textwidth]{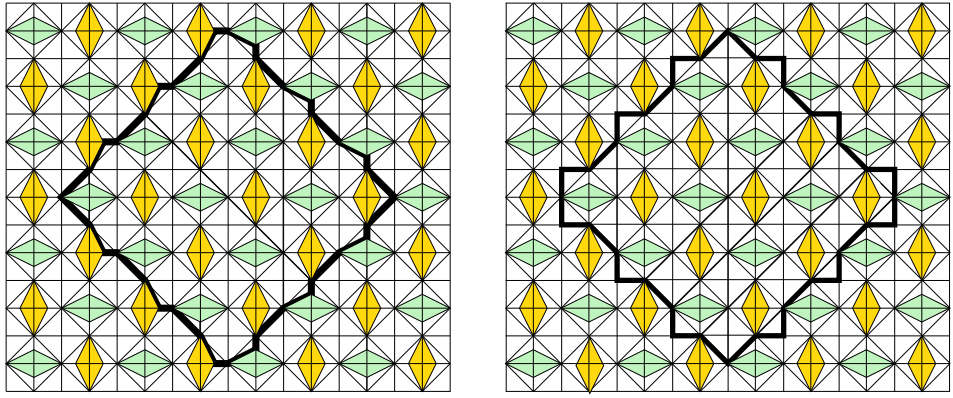}
\caption{The region $S^{(2)}_6$ (left) and the region $S^{(3)}_6$ (right).}
\label{Fig25}
\end{figure}

\begin{figure}\centering
\includegraphics[width=0.90\textwidth]{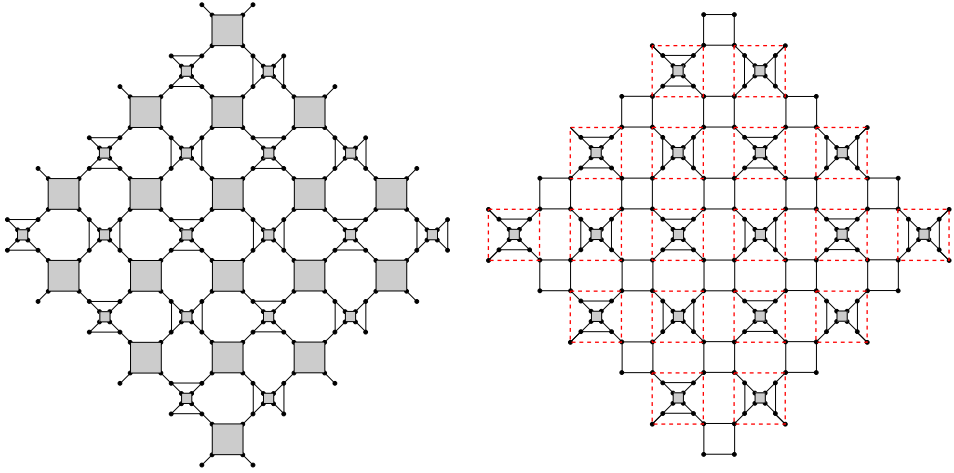}
\caption{The dual graph of $S^{(2)}_6$ (left) and the dual graph of $S^{(3)}_6$ (right).}
\label{Fig26}
\end{figure}

We get the following formulas for the numbers tilings of $S^{(2)}_n$ and $S^{(3)}_n$ by applying Theorem \ref{7gen}.

\begin{corollary}\label{power7cor}
(a) For any nonnegative integer $n$
\begin{equation}\T(S^{(2)}_{4n})=7^{n(2n+1)}2^{12n^2-2n},\  \T(S^{(2)}_{4n+1})=5\cdot7^{n(2n+2)}2^{12n^2+4n},\notag\end{equation}
\begin{equation}\T(S^{(2)}_{4n+2})=7^{n(2n+3)+1}2^{12n^2+10n+2},\ \T(S^{(2)}_{4n+3})=7^{n(2n+4)+2}2^{12n^2+16n+5}.\ \ \notag\end{equation}
(b) For any $n\geq 0$
\begin{equation}\T(S^{(3)}_{4n})=7^{n(2n-1)}2^{n(12n+2)},\ \ \T(S^{(3)}_{4n+1})=7^{2n^2}2^{n(12n+8)+1},\notag\end{equation}
\begin{equation}\T(S^{(3)}_{4n+2})=7^{n(2n+1)}2^{n(12n+14)+4},\ \ \T(S^{(3)}_{4n+3})=5\cdot7^{2n(n+1)}2^{n(12n+20)+8}.\notag\end{equation}
\end{corollary}

\begin{proof}
(a) Consider the dual graph of $S^{(2)}_n$ (see the left picture in Figure \ref{Fig26}). We do again the process in the proof of Theorem \ref{power7thm}. In particular,  we apply the replacement in Spider Lemma \ref{spider}(a) to all $n^2$ shaded cells, and replace all multiple edges by corresponding single edges in the resulting graph using Lemma \ref{ER}. We get a graph isomorphic to the weighted Aztec diamond graph $AD_n(\wt_{B_2})$, where
\[B_2=\begin{bmatrix}
1/2&3/2&1/2&1/2\\
3/2&1/2&1/2&1/2\\
1/2&1/2&3/2&1/2\\
1/2&1/2&1/2&3/2
\end{bmatrix}.\]
By Lemmas \ref{ER} and \ref{spider}, we obtain
  \begin{equation}\label{corpower7}
  \T(S^{(2)}_{n})=2^{n^2}\M(AD_{n}(\wt_{B_2})).
  \end{equation}
We get the statement by applying  Theorem \ref{7gen} (for $a=1/2$,  $b=3/2$ and $c=d=1/2$) and (\ref{corpower7}).

(b) Consider the dual graph of $S^{(3)}_n$ (illustrated by the right picture in Figure \ref{Fig26}). Each of the shaded cells gives us a chance to apply Spider Lemma \ref{spider}(a). After replacing all multiple edges (arising from the replacement of the Spider Lemma) by single edges as in Lemma \ref{ER}, we have another chance to apply the Spider Lemma \ref{spider}(a) to a new cell. The cell in the first application of Spider Lemma has cell-factor 2 (it has all edges weighted by 1), the cell in the second application of Spider Lemma has cell-factor $5/2$ (it has edges with weights $1/2,$ $3/2,$ $1/2,$ $3/2$ in cyclic order). The resulting graph is exactly a version of the weighted Aztec diamond $AD_n(\wt_{B_3})$, and
\begin{equation}\label{corpower7'}
\frac{\T(S^{(3)}_n)}{\M(AD_{n}(\wt_{B_3}))}=
\begin{cases}
2^{2k^2}(5/2)^{2k^2} &\text{if $n=2k$;}\\
2^{(k+1)^2+k^2}(5/2)^{(k+1)^2+k^2} &\text{if $n=2k+1$,}
\end{cases}\end{equation}
 where
\[B_3=\begin{bmatrix}
1/5&3/5&1&1\\
3/5&1/5&1&1\\
1&1&3/5&1/5\\
1&1&1/5&3/5
\end{bmatrix}.\]
Finally, Theorem \ref{7gen} (for $a=1/5$, $b=3/5$ and $c=d=1$) and (\ref{corpower7'}) imply part (b).
  \end{proof}

\begin{figure}\centering
\includegraphics[width=0.90\textwidth]{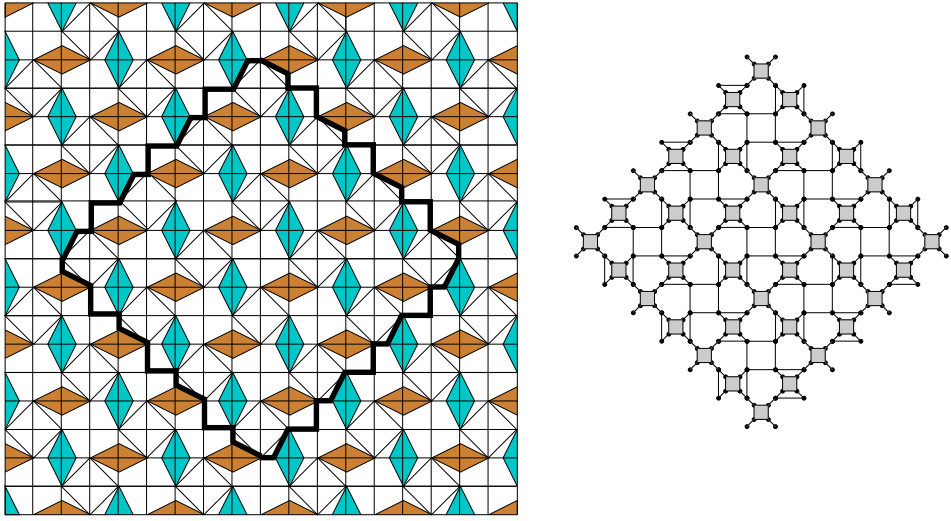}
\caption{The region $S^{(4)}_7$ (left) and its dual graph (right).}
\label{Fig27}
\end{figure}

We modify the lattice in the definition of $S^{(2)}_n$ and $S^{(3)}_n$ by removing the boundaries of all unit diamonds $\mathbb{D}_1(4k,4l+1)$ and $\mathbb{D}_1(4k+2,4l+3)$, for any two integer numbers $k$ and $l$. We are interested in a new family of regions on the resulting lattice defined as follows. The region of order $n$ consists of all elementary regions lying entirely or partially inside the diamond $\mathbb{D}_n(1,1)$ together with all elementary regions that have an edge resting on the boundary of  $\mathbb{D}_n(1,1)$ (the region of order $7$ is shown by the left picture in Figure \ref{Fig27}). Denote by $S^{(4)}_n$ the region of order $n$. The number tilings of $S^{(4)}_n$ can be obtained by the following theorem.

\begin{theorem}\label{p31}For any nonnegative integer $n$

\begin{align}
\T(S^{(4)}_{4n})&=2^{n(2n-1)}5^{n(2n-1)}31^{n(2n+1)},\notag\\
\T(S^{(4)}_{4n+1})&=2^{2n^2}5^{2n^2}31^{2n(n+1)},\notag\\
\T(S^{(4)}_{4n+2})&=2^{n(2n+1)}5^{n(2n+1)}31^{(n+1)(2n+1)},\notag\\
\T(S^{(4)}_{4n+3})&=2^{2n(n+1)}5^{2n(n+1)}31^{2(n+1)^2}.\notag\end{align}
\end{theorem}
\begin{proof}
Apply the replacement in Spider Lemma \ref{spider}(a) to all shaded cells in the dual graph of $S^{(4)}_n$ (see the right picture in Figure \ref{Fig27}), and apply Lemma \ref{ER} to all multiple edges. We get
\begin{equation}\label{p31eq1}
\frac{\T(S^{(4)}_{n})}{\M(AD_n(\wt_{B_4}))}=
\begin{cases}2^{3k^2} &\text{if $n=2k$;}\\
2^{(k+1)(3k+1)} &\text{if $n=2k+1$,}
\end{cases}
\end{equation}
where
\[B_4=\begin{bmatrix}
1/2&1/2&1/2&3/2\\
1/2&1/2&3/2&1/2\\
3/2&1/2&1&1\\
1/2&3/2&1&1
\end{bmatrix}.\]

One verifies easily that
\begin{equation}\de(B_4)=\begin{bmatrix}
1&3/5&1/5&1\\
1/5&1/2&1/2&3/5\\
3/5&1/2&1/2&1/5\\
1&1/5&3/5&1
\end{bmatrix},\notag\end{equation}
\begin{equation}\de^{(2)}(B_4)=\begin{bmatrix}
50/31&50/31&10/31&30/31\\
50/31&50/31&30/31&10/31\\
30/31&10/31&25/31&25/31\\
10/31&30/31&25/31&25/31
\end{bmatrix},\notag\end{equation}
\begin{equation}\de^{(3)}(B_4)=\begin{bmatrix}
31/100&93/100&31/100&31/100\\
31/100&31/50&31/50&93/100\\
93/100&31/50&31/50&31/100\\
31/100&31/100&93/100&31/100
\end{bmatrix},\notag\end{equation}
\begin{equation}\de^{(3)}(B_4)=\begin{bmatrix}
31/100&93/100&31/100&31/100\\
31/100&31/50&31/50&93/100\\
93/100&31/50&31/50&31/100\\
31/100&31/100&93/100&31/100
\end{bmatrix},\notag\end{equation}
and
\begin{equation}\label{p31m}\de^{(4)}(B_4)=\frac{40}{31}B_4.\end{equation}

By counting the cell-factors and applying Reduction Theorem \ref{reduction}, we get the following four equalities:
\begin{equation}\label{p31eq2}
\frac{\M(AD_{2n}(\wt_{B_4}))}{\M(AD_{2n-1}(\wt_{\de(B_4)}))}=(1/2)^{n^2}2^{n^2}(5/2)^{2n^2},
\end{equation}
\begin{equation}\label{p31eq3}
\frac{\M(AD_{2n-1}(\wt_{\de(B_4)}))}{\M(AD_{2n-2}(\wt_{\de^{(2)}(B_4)}))}=(31/50)^{(2n-1)^2},
\end{equation}
\begin{equation}\label{p31eq4}
\frac{\M(AD_{2n-2}(\wt_{\de^{(2)}(B_4)}))}{\M(AD_{2n-3}(\wt_{\de^{(3)}(B_4)}))}=\left(\frac{50^2}{2\cdot 31^2}\right)^{(n-1)^2}\left(\frac{2\cdot50^2}{31^2}\right)^{(n-1)^2}\left(\frac{2\cdot50^2}{31^2}\right)^{2(n-1)^2},
\end{equation}
\begin{equation}\label{p31eq5}
\frac{\M(AD_{2n-3}(\wt_{\de^{(3)}(B_4)}))}{\M(AD_{2n-4}(\wt_{\de^{(4)}(B_4)}))}=(5\cdot(31/100)^2)^{(2n-3)^2}.
\end{equation}
Lemma \ref{mulstar} and equality (\ref{p31m}) imply
\begin{equation}\label{p31eq6}
\frac{\M(AD_{2n-4}(\wt_{\de^{(4)}(B_4)}))}{\M(AD_{2n-4}(\wt_{B_4}))}=(40/31)^{(2n-3)(2n-4)}.
\end{equation}
By the five equalities (\ref{p31eq2})--(\ref{p31eq6}) above, we obtain the following recurrence
\begin{equation}\label{p31eq7}
\frac{\M(AD_{2n}(\wt_{B_4}))}{\M(AD_{2n-4}(\wt_{B_4}))}=2^{-10n+9}5^{2n-3}31^{2n-1}.
\end{equation}
In the same fashion,  we get a similar recurrence for the odd-order Aztec diamond graphs
\begin{equation}\label{p31eq8}
\frac{\M(AD_{2n+1}(\wt_{B_4}))}{\M(AD_{2n-3}(\wt_{B_4}))}=2^{-10n+2}5^{2n-2}31^{2n}.
\end{equation}
Finally, the theorem follows from two recurrences (\ref{p31eq7}) and (\ref{p31eq8}) together with (\ref{p31eq1}).
  \end{proof}

Our next target is to create a new lattice by periodically applying the node replacements with a more complicated period.  We start with the square lattice with all second diagonals drawn in (see the right picture in Figure \ref{Fig6}). Apply the first node replacement to all nodes $(8k+5,4l)$, $(8k+7,4l)$, $(8k,4l+1)$, $(8k+2,4l+1)$, $(8k+3,4l+2)$, $(8k+5,4l+2)$, $(8k+6,4l+3)$ and $(8k,4l+3)$, for any two integer numbers $k$ and $l$; and apply the second node replacement to all remaining nodes which have the $x$- and $y$-coordinates of different parity. Next, we consider a variant of fortress on the resulting lattice. The region of order $n$ is defined similarly to the regions $F_n$ and $S^{(3)}_n$ based on the diamond $\mathbb{D}_n(0,0)$ (illustrated by Figure \ref{Fig28}). Denote by $Q_n$ the region of order $n$. We will show in the following theorem that the number tilings of $Q_n$ is (or nearly is) the product of a perfect power of $3$ and a perfect power of $29$.

\begin{figure}\centering
\includegraphics[width=0.5\textwidth]{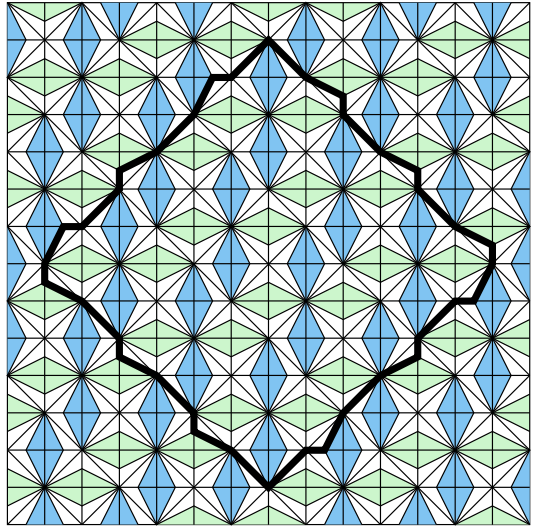}
\caption{The region $Q_6$.}
\label{Fig28}
\end{figure}

\begin{theorem}\label{power29} For  any $q\geq 0$
\begin{align}
\T(Q_{4q})&=3^{4q^2}29^{4q^2},\notag\\
\T(Q_{8q+2})&=3^{2q(8q+4)}29^{(4q+1)^2},\notag\\
\T(Q_{8q+6})&=3^{8(q+1)(2q+1)+2}29^{(4q+3)^2},\notag\\
\T(Q_{4q+1})&=2\cdot3^{2q(2q+1)}29^{2q(2q+1)},\notag\\
\T(Q_{4q+3})&=5\cdot3^{(2q+1)(2q+2)}29^{(2q+1)(2q+2)}.\notag
\end{align}
\end{theorem}

Before presenting the proof of Theorem \ref{power29}, we consider a $4\times 2n$ weight pattern
\begin{equation}
C=\begin{bmatrix}
a_1&b_1&a_2&b_2&\dotsc&a_n&b_n\\
d_1&c_1&d_2&c_2&\dotsc&d_n&c_n\\
d_1/\Delta_1&c_1/\Delta_1&d_2/\Delta_2&c_2/\Delta_2&\dotsc&d_n/\Delta_n&c_n/\Delta_n\\
a_1/\Delta_1&b_1/\Delta_1&a_2/\Delta_2&b_2/\Delta_2&\dotsc&a_n/\Delta_n&b_n/\Delta_n\\
\end{bmatrix},
\notag\end{equation}
where $a_i,b_i,c_i,d_i$ are  positive numbers and where $\Delta_i=a_ic_i+b_id_i$, for $i=1,2,\dotsc,n$. One can write $C$ as a block matrix as follows.
\begin{equation}
C=\begin{bmatrix}
A_1&A_2&\dotsc&A_n
\end{bmatrix},
\notag\end{equation}
where
\[A_i=\begin{bmatrix}
a_i&b_i\\
d_i&c_i\\
d_i/\Delta_i&c_i/\Delta_i\\
a_i/\Delta_i&b_i/\Delta_i
\end{bmatrix}.
\]
The matching generating function of $AD_n(\wt_C)$ is given by the theorem below.

\begin{theorem}\label{power29gen}
(a) For any positive integer $k$
\begin{align}\label{29recure}
\M(AD_{2k}(\wt_{C}))=&\prod_{i=1}^ka_i^{k-i+1}b_{2k+1-i}^{k+1-i}c_{2k+1-i}^{k-i}d_i^{k-i}\notag\\
&\times\prod_{i=1}^{2k}\Delta_{i}^{-k}\cdot \prod_{i=1}^{k-1}(\Delta_i+\Delta_{i+1})^{i}\cdot \prod_{i=1}^k(\Delta_{2k-i}+\Delta_{2k-i+1})^{i}.
\end{align}

(b) For nonnegative integer $k$
\begin{align}\label{29recuro}
\M(AD_{2k+1}(\wt_{C}))=&\prod_{i=1}^{k+1}a_i^{k-i+1}b_{2k+2-i}^{k+1-i}c_{2k+2-i}^{k+1-i}d_i^{k+1-i}\notag\\
&\times\Delta_{k+1}\prod_{i=1}^{2k}\Delta_{i}^{-k}\cdot\prod_{i=1}^{k}(\Delta_i+\Delta_{i+1})^{i}\cdot\prod_{i=1}^k(\Delta_{2k+1-i}+\Delta_{2k+2-i})^{i}.
\end{align}
\end{theorem}

\begin{proof}
(a)
One can verify that $\de(C)$, $\de^{(2)}(C)$, $\de^{(3)}(C)$ and $\de^{(4)}(C)$ are four block matrices of size $4\times 2n$ presented as below.
\begin{align}
\de(C)&=\begin{bmatrix}B_1& B_2&\dotsc &B_n\end{bmatrix}\notag\\
\de^{(2)}(C)&=\begin{bmatrix}C_1&C_2& \dotsc &C_n\end{bmatrix}\notag\\
\de^{(3)}(C)&=\begin{bmatrix}D_1&D_2&\dots&D_n\end{bmatrix}\notag\\
\de^{(4)}(C)&=\begin{bmatrix}E_1&E_2&\dots&E_n\end{bmatrix},\notag
\end{align}
where
\[B_i=\begin{bmatrix}
a_i/\Delta_i&b_{i+1}/\Delta_{i+1}\\
a_i&b_{i+1}\\
d_i&c_{i+1}\\
d_i/\Delta_i&c_{i+1}/\Delta_{i+1}
\end{bmatrix},\]
\[C_i=\begin{bmatrix}
\frac{\Delta_{i+1}}{b_{i+1}(\Delta_{i}+\Delta_{i+1})}&\frac{\Delta_{i+1}}{a_{i+1}(\Delta_{i+1}+\Delta_{i+2})}\\
\frac{\Delta_{i+1}}{c_{i+1}(\Delta_{i}+\Delta_{i+1})}&\frac{\Delta_{i+1}}{d_{i+1}(\Delta_{i+1}+\Delta_{i+2})}\\
\frac{\Delta_{i}\Delta_{i+1}}{c_{i+1}(\Delta_{i}+\Delta_{i+1})}&\frac{\Delta_{i+1}\Delta_{i+2}}{d_{i+1}(\Delta_{i+1}+\Delta_{i+2})}\\
\frac{\Delta_{i}\Delta_{i+1}}{b_{i+1}(\Delta_{i}+\Delta_{i+1})}&\frac{\Delta_{i+1}\Delta_{i+2}}{a_{i+1}(\Delta_{i+1}+\Delta_{i+2})}
\end{bmatrix},\]
\[D_i=\begin{bmatrix}
\frac{a_{i+1}c_{i+1}d_{i+1}(\Delta_{i+1}+\Delta_{i+2})}{\Delta_{i+1}^2}&\frac{b_{i+2}c_{i+2}d_{i+2}(\Delta_{i+1}+\Delta_{i+2})}{\Delta_{i+2}^2}\\
\frac{a_{i+1}c_{i+1}d_{i+1}(\Delta_{i+1}+\Delta_{i+2})}{\Delta_{i+1}^2\Delta_{i+2}}&\frac{b_{i+2}c_{i+2}d_{i+2}(\Delta_{i+1}+\Delta_{i+2})}{\Delta_{i+1}\Delta_{i+2}^2}\\
\frac{a_{i+1}b_{i+1}d_{i+1}(\Delta_{i+1}+\Delta_{i+2})}{\Delta_{i+1}^2\Delta_{i+2}}&\frac{a_{i+2}b_{i+2}c_{i+2}(\Delta_{i+1}+\Delta_{i+2})}{\Delta_{i+1}\Delta_{i+2}^2}\\
\frac{a_{i+1}b_{i+1}d_{i+1}(\Delta_{i+1}+\Delta_{i+2})}{\Delta_{i+1}^2}&\frac{a_{i+2}b_{i+2}c_{i+2}(\Delta_{i+1}+\Delta_{i+2})}{\Delta_{i+2}^2}
\end{bmatrix},\]
\[
E_i=\begin{bmatrix}
\frac{\Delta_{i+1}\Delta_{i+2}^3}{b_{i+2}c_{i+2}d_{i+2}(\Delta_{i+1}+\Delta_{i+2})^2}&\frac{\Delta_{i+2}^3\Delta_{i+3}}{a_{i+2}c_{i+2}d_{i+2}(\Delta_{i+2}+\Delta_{i+3})^2}\\
\frac{\Delta_{i+1}\Delta_{i+2}^3}{a_{i+2}b_{i+2}c_{i+2}(\Delta_{i+1}+\Delta_{i+2})^2}&\frac{\Delta_{i+2}^3\Delta_{i+3}}{a_{i+2}b_{i+2}d_{i+2}(\Delta_{i+2}+\Delta_{i+3})^2}\\
\frac{\Delta_{i+1}\Delta_{i+2}^2}{b_{i+2}b_{i+2}c_{i+2}(\Delta_{i+1}+\Delta_{i+2})^2}&\frac{\Delta_{i+2}^2\Delta_{i+3}}{a_{i+2}b_{i+2}d_{i+2}(\Delta_{i+2}+\Delta_{i+3})^2}\\
\frac{\Delta_{i+1}\Delta_{i+2}^2}{b_{i+2}c_{i+2}d_{i+2}(\Delta_{i+1}+\Delta_{i+2})^2}&\frac{\Delta_{i+2}^2\Delta_{i+3}}{a_{i+2}c_{i+2}d_{i+2}(\Delta_{i+2}+\Delta_{i+3})^2}
\end{bmatrix},
\]
for $i=1,2,\dotsc,n$ (the subscripts here are interpreted modulo $n$).

Apply Reduction Theorem \ref{reduction}, we have four recurrences below.
\begin{align}\label{gf1}
\M(AD_{2k}(\wt_C))&=\M(AD_{2k-1}(\wt_{\de(A)}))\prod_{i=1}^{2k}\Delta_{i}^{k_1}\Delta_{i}^{-k}\notag\\
&=M(AD_{2k-1}(wt_{d(C)})).
\end{align}
\begin{equation}\label{gf2}
\frac{\M(AD_{2k-1}(\wt_{d(C)}))}{\M(AD_{2k-2}(\wt_{\de^{(2)}(C)}))}=\prod_{i=1}^{2k-1}\left(\frac{a_ib_{i+1}(\Delta_i+\Delta_{i+1})}{\Delta_{i}\Delta_{i+1}}\right)^{k}
\left(\frac{c_{i+1}d_i(\Delta_i+\Delta_{i+1})}{\Delta_{i}\Delta_{i+1}}\right)^{k-1}.
\end{equation}
\begin{equation}\label{gf3}
\frac{\M(AD_{2k-2}(\wt_{\de^{(2)}(C)}))}{\M(AD_{2k-3}(\wt_{\de^{(3)}(C)}))}=\prod_{i=1}^{2k-2}\left(\frac{\Delta_i\Delta_{i+1}^6\Delta_{i+2}}
{a_{i+1}^2b_{i+1}^2c_{i+1}^2d_{i+1}^2(\Delta_{i}+\Delta_{i+1})^2(\Delta_{i+1}+\Delta_{i+2})^2}\right)^{k-1}.
\end{equation}
\begin{align}\label{gf4}
\frac{\M(AD_{2k-3}(\wt_{\de^{(3)}(C)}))}{\M(AD_{2k-4}(\wt_{\de^{(4)}(C)}))}=&\prod_{i=1}^{2k-3}\left(\frac{a_{i+1}c_{i+1}d_{i+1}b_{i+2}c_{i+2}d_{i+2}(\Delta_{i+1}+\Delta_{i+2})^3}{\Delta_{i+1}^3\Delta_{i+2}^3}\right)^{k-1}\notag\\
&\times\left(\frac{a_{i+1}b_{i+1}d_{i+1}a_{i+2}b_{i+2}c_{i+2}(\Delta_{i+1}+\Delta_{i+2})^3}{\Delta_{i+1}^3\Delta_{i+2}^3}\right)^{k-2}.
\end{align}

Divide the weight matrix of $AD_{2k-4}(\wt_{\de^{(4)}(C)})$  into $2k-3$ parts (by columns) as in Lemma \ref{mulstar}(a), and multiply all entries of the $i$-th part by $\dfrac{(\Delta_{i+1}+\Delta_{i+2})^2}{\Delta_{i+1}\Delta_{i+2}}$, for $i=1,2,\dots,2k-3$. Divide the resulting matrix into $2k-4$ parts (by columns) as in Lemma \ref{mulstar}(b), and multiply all entries of the $j$-th part by $\dfrac{a_{j+2}b_{j+2}c_{j+2}d_{j+2}}{\Delta_{j+2}^2}$, for $j=1,2,\dots,2k-4$. We get the weight matrix of the Aztec diamond $AD_{2k-4}(\wt_{\phi(C)})$, where
\begin{equation}
\phi(C)=\begin{bmatrix}A_3&A_4&\dots&A_{n-2}\end{bmatrix},
\end{equation}
and where the operator $\phi$ is defined as in the proof of Theorem \ref{weighted}. Therefore, Lemma \ref{mulstar} implies
{\small\begin{equation}\label{gf5}
\frac{\M(AD_{2k-4}(\wt_{\de^{(4)}(C)}))}{\M(AD_{2k-4}(\wt_{\phi(C)}))}=\prod_{i=1}^{2k-3}\left(\frac{(\Delta_{i+1}+\Delta_{i+2})^2}{\Delta_{i+1}\Delta_{i+2}}\right)^{-(2k-4)}\prod_{i=1}^{2k-4}
\left(\frac{a_{i+2}b_{i+2}c_{i+2}d_{i+2}}{\Delta_{i+2}^2}\right)^{-(2k-3)}.
\end{equation}}
By (\ref{gf1})--(\ref{gf5}), we obtain
\begin{align}\label{gf6}
\frac{\M(AD_{2k}(\wt_{(C)}))}{\M(AD_{2k-4}(\wt_{\phi(C)}))}=&a_1^{k}a_2^{k-1}b_{2k-1}^{k-1}b_{2k}^kc_{2k-1}^{k-2}c_{2k}^{k-1}d_1^{k-1}d_2^{k-2}\Delta_1^{-k}\Delta_2^{-k}\Delta_{2k-1}^{-k}\Delta_{2k}^{-k}\prod_{i=3}^{2n-2}(\Delta_{i}^{-2})\notag\\
&\times(\Delta_1+\Delta_2)(\Delta_{2k-1}+\Delta_{2k})\prod_{i=1}^{2k-3}(\Delta_{i+1}+\Delta_{i+2}).
\end{align}
Repeated application of (\ref{gf6}) yields (\ref{29recure}).

\medskip
(b) Assume that $n=2k+1$. In the same fashion, we get
\begin{align}\label{gf6o}
\frac{\M(AD_{2k+1}(\wt_{(C)}))}{\M(AD_{2k-3}(\wt_{\phi(C)}))}=&a_1^{k}a_2^{k-1}b_{2k}^{k-1}b_{2k+1}^kc_{2k}^{k-1}c_{2k+1}^{k}d_1^{k}d_2^{k-1}\Delta_1^{-k}\Delta_2^{-k}\Delta_{2k}^{-k}\Delta_{2k+!}^{-k}\notag\\
&\times\prod_{i=3}^{2n-1}(\Delta_{i}^{-2}).(\Delta_1+\Delta_2)(\Delta_{2k}+\Delta_{2k+1})\prod_{i=1}^{2k-2}(\Delta_{i+1}+\Delta_{i+2}).
\end{align}
Again, $(\ref{29recuro})$ is obtained by applying the recurrence (\ref{gf6o}) repeatedly.
\end{proof}

We are now ready to present the proof of Theorem \ref{power29} by using Theorem \ref{power29gen}.

\begin{figure}\centering
\includegraphics[width=0.5\textwidth]{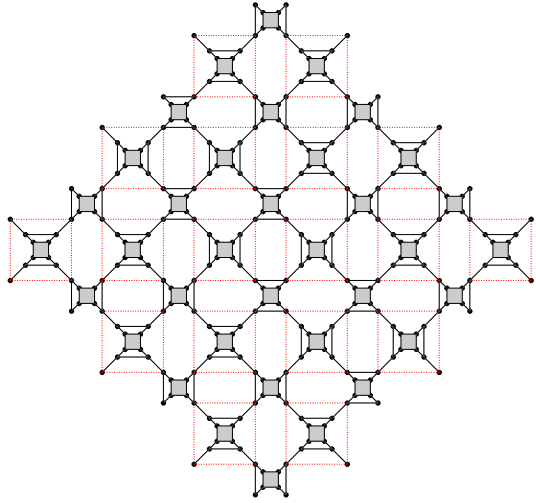}
\caption{The dual graph of $Q_6$.}
\label{Fig29}
\end{figure}

\begin{proof}[Proof of Theorem \ref{power29}]
Consider the dual graph of $Q_n$ (see Figure \ref{Fig29}). Apply Spider Lemma \ref{spider}(a) \textit{twice} to all shaded cells inside the dotted squares (and apply Lemma \ref{ER} to all multiple edges if needed) by the same way as we did in the proof of Corollary \ref{power7cor}(b). The cell-factors in the first application of the Spider Lemma are $2$, and the cell-factors in the second application of the Spider Lemma are $5/2$. Then apply the Spider Lemma \ref{spider}(a) to all other shaded cells (the cell-factors are all $2$). The process gives us a graph isomorphic to $AD_n(\wt_{C_0})$, where
\begin{equation}
C_0=\begin{bmatrix}
1/5&3/5&3/2&1/2&3/5&1/5&1/2&3/2\\
3/5&1/5&1/2&3/2&1/5&3/5&3/2&1/2\\
3/2&1/2&1/5&3/5&1/2&3/2&3/5&1/5\\
1/2&3/2&3/5&1/5&3/2&1/2&1/5&3/5
\end{bmatrix}.
\notag\end{equation}
By enumerating the shaded cells in each type above, we have
\begin{equation}\label{power29eq}
\frac{\T(Q_n)}{\M(AD_n(\wt_{C_0}))}=\begin{cases} 2^{2k^2}(5/2)^{2k^2}2^{2k^2}&\text{if $n=2k$;}\\
2^{(k+1)^2+k^2}(5/2)^{(k+1)^2+k^2}2^{2k(k+1)}&\text{ if $n=2k+1$.}
\end{cases}
\end{equation}
Therefore, the theorem follows (\ref{power29eq}) and Theorem \ref{power29gen}  by specializing

$\quad\quad a_{4i+1}=c_{4i+1}=1/5$ and $b_{4i+1}=d_{4i+1}=3/5$, for $i=0,1,\dots,\lfloor\frac{n-1}{4}\rfloor$;

$\quad\quad a_{4i+2}=c_{4i+2}=3/2$ and $b_{4i+2}=d_{4i+2}=1/2$, for $i=0,1,\dots,\lfloor\frac{n-2}{4}\rfloor$;

$\quad\quad a_{4i+3}=c_{4i+3}=3/5$ and $b_{4i+3}=d_{4i+3}=1/5$, for $i=0,1,\dots,\lfloor\frac{n-3}{4}\rfloor$;

$\quad\quad a_{4i+4}=c_{4i+4}=1/2$ and $b_{4i+4}=d_{4i+4}=3/2$, for $i=0,1,\dots,\lfloor\frac{n-4}{4}\rfloor$.
  \end{proof}

One readily sees that the weight pattern $C$ in Theorem \ref{power29gen} is also a generalization of the weight patterns of  $AD(d_1,\dotsc,d_m)$ and $\overline{AD}(d_1,\dotsc,d_m)$ in the proof of Theorem \ref{main}. Therefore, Theorem \ref{power29gen} can be viewed as a common multi-parameter generalization of Theorems \ref{main} and \ref{power29}.

\medskip
We conclude this section by presenting one more multi-parameter generalization of \ref{power29}. Let us consider a new weight pattern
\begin{equation}
N=\begin{bmatrix}
a&b&x&y&b&a&y&x\\
d&c&t&z&c&d&z&t\\
x&y&a&b&y&x&b&a\\
t&z&c&d&z&t&c&d
\end{bmatrix},
\end{equation}
for any positive numbers $a,$ $b,$ $c,$ $d,$ $x,$ $y,$ $z,$ $t$. We get the weight pattern $C_0$ in the proof of Theorem \ref{power29} by specializing $a=c=1/5$, $b=d=3/5$, $x=z=3/2$ and $y=t=1/2$.
\begin{theorem}
The values of $\M(AD_m(wt_N))$ are given by the following recurrences for $n\geq2$
\begin{align}
\frac{\M(AD_{4n}(\wt_{N}))}{\M(AD_{4n-8}(\wt_{N}))}=&\Delta_0^{8n-8}(ad+xt)^{2n-2}(bc+yz)^{2n-4}C_1^{2n}C_2^{2n-2}\notag\\
&\times a^{2n-1}b^{2n-3}c^{2n-3}d^{2n-1}x^{2n-1}y^{2n-3}z^{2n-3}t^{2n-1},\notag
\end{align}
\begin{align}
\frac{\M(AD_{4n+1}(\wt_{N}))}{\M(AD_{4n-7}(\wt_{N}))}=&\Delta_0^{8n-6}(ad+xt)^{2n-2}(bc+yz)^{2n-3}C_1^{2n}C_2^{2n-1}\notag\\
&a^{2n-1}b^{2n-2}c^{2n-2}d^{2n-1}x^{2n-1}y^{2n-2}z^{2n-2}t^{2n-1},\notag
\end{align}
\begin{align}
\frac{\M(AD_{4n+2}(\wt_{N}))}{\M(AD_{4n-6}(\wt_{N}))}=&\Delta_0^{8n-4}(ad+xt)^{2n-2}(bc+yz)^{2n-2}C_1^{2n}C_2^{2n}\notag\\
&\times a^{2n-1}b^{2n-1}c^{2n-1}d^{2n-1}x^{2n-1}y^{2n-1}z^{2n-1}t^{2n-1},\notag
\end{align}
\begin{align}
\frac{\M(AD_{4n+3}(\wt_{N}))}{\M(AD_{4n-5}(\wt_{N}))}=&\Delta_0^{8n-2}(ad+xt)^{2n-1}(bc+yz)^{2n-2}C_1^{2n+1}C_2^{2n}\notag\\
&\times a^{2n}b^{2n-1}c^{2n-1}d^{2n}x^{2n}y^{2n-1}z^{2n-1}t^{2n},\notag
\end{align}
and the initial values\\
$\M(AD_{0}(\wt_{N}))=1$,\\
$\M(AD_{1}(\wt_{N}))=\Delta_1$,\\
$\M(AD_{2}(\wt_{N}))=C_3$,\\
$\M(AD_{3}(\wt_{N}))=ad\Delta_0\Delta_2C_1$,\\
$\M(AD_{4}(\wt_{N}))=adxt\Delta_0^2C_1^2$,\\
$\M(AD_{5}(\wt_{N}))=abcdxt\Delta_0^3C_1^2C_2$,\\
$\M(AD_{6}(\wt_{N}))=abcdxyzt\Delta_0^4\Delta_2C_1^2C_2^2$,\\
$\M(AD_{7}(\wt_{N}))=a^2bcd^2x^2yzt^2\Delta_0^6\Delta_1C_1^3C_2^2,$\\
where $\Delta_1=ac+bd$, $\Delta_2=xz+yt$, $\Delta_0=\Delta_1+\Delta_2$, $C_1=xt\Delta_1^2+ad\Delta_2^2$, $C_2=yz\Delta_1^2+bc\Delta_2^2$, $C_3=ac\Delta_2^2+yt\Delta_1^2$, and $C_4=bd\Delta_2^2+xz\Delta_1^2$ .
\end{theorem}

\begin{proof}[Sketch of proof]
We prove only the first recurrence, the other ones can be treated similarly.

By the Reduction Theorem and the definition of the operator $\de$ we have four recurrences as below.
\begin{equation}\label{star1}
\frac{\M(AD_{4n}(\wt_{N}))}{\M(AD_{4n-1}(\wt_{\de(N)}))}=\Delta_1^{8n^2}\Delta_2^{8n^2}.
\end{equation}
\begin{align}\label{star2}
\frac{\M(AD_{4n-1}(\wt_{\de(N)}))}{\M(AD_{4n-2}(\wt_{\de^{(2)}(N)}))}=&\left(\frac{C_1}{\Delta_1^2\Delta_2^2}\right)^{(4n-1)n}\left(\frac{C_2}{\Delta_1^2\Delta_2^2}\right)^{(4n-1)(n-1)}\notag\\
&\times\left(\frac{C_3}{\Delta_1^2\Delta_2^2}\right)^{(4n-1)n}\left(\frac{C_4}{\Delta_1^2\Delta_2^2}\right)^{(4n-1)n}.
\end{align}
\begin{align}\label{star3}
\frac{\M(AD_{4n-2}(\wt_{\de^{(2)}(N)}))}{\M(AD_{4n-3}(\wt_{\de^{(3)}(N)}))}=&\left(\frac{\Delta_0\Delta_1^3\Delta_2^5ad}{C_1C_3C_4}\right)^{(4n-2)n}\left(\frac{\Delta_0\Delta_1^5\Delta_2^3xt}{C_1C_3C_4}\right)^{(4n-2)n}\notag\\
&\times\left(\frac{\Delta_0\Delta_1^3\Delta_2^5bc}{C_2C_3C_4}\right)^{(4n-2)(n-1)}\left(\frac{\Delta_0\Delta_1^5\Delta_2^3yz}{C_2C_3C_4}\right)^{(4n-2)(n-1)}.
\end{align}
\begin{align}\label{star4}
\frac{\M(AD_{4n-3}(\wt_{\de^{(3)}(N)}))}{\M(AD_{4n-4}(\wt_{\de^{(4)}(N)}))}=&\left(\frac{C_3C_4C_1^3}{\Delta_0^2\Delta_1^6\Delta_2^6adxt}\right)^{(4n-3)n}\left(\frac{C_3^3C_4^2}{\Delta_0^2\Delta_1^6\Delta_2^6acyt}\right)^{(4n-3)(n-1)}\notag\\
&\times\left(\frac{C_3^2C_4^3}{\Delta_0^2\Delta_1^6\Delta_2^6bdxz}\right)^{(4n-3)(n-1)}\left(\frac{C_3C_4C_y^3}{\Delta_0^2\Delta_1^6\Delta_2^6bcyz}\right)^{(4n-3)(n-1)}.
\end{align}
Moreover, we can verify that matrix $\de^{(4)}(N)$ is given by
\begin{align}&\de^{(4)}(N)=\notag\\
&\begin{bmatrix}
\frac{\Delta_{0}\Delta_1^4\Delta_2^3axt}{C_1^2C_3}\frac{\Delta_{0}\Delta_1^4\Delta_2^3ayt}{C_3^2C_4}\frac{\Delta_{0}\Delta_1^3\Delta_2^4acy}{C_3^2C_4}\frac{\Delta_{0}\Delta_1^3\Delta_2^4bcy}{C_3C_2^2}
\frac{\Delta_{0}\Delta_1^4\Delta_2^3byz}{C_2^2C_4}\frac{\Delta_{0}\Delta_1^4\Delta_2^3bxz}{C_3C_4^2}\frac{\Delta_{0}\Delta_1^3\Delta_2^4bdx}{C_3C_4^2}\frac{\Delta_{0}\Delta_1^3\Delta_2^4adx}{C_4C_1^2}\\
\frac{\Delta_{0}\Delta_1^4\Delta_2^3dxt}{C_1^2C_4}\frac{\Delta_{0}\Delta_1^4\Delta_2^3dxz}{C_3C_4^2}\frac{\Delta_{0}\Delta_1^3\Delta_2^4bdz}{C_3C_4^2}\frac{\Delta_{0}\Delta_1^3\Delta_2^4bcz}{C_4C_2^2}
\frac{\Delta_{0}\Delta_1^4\Delta_2^3cyz}{C_3C_2^2}\frac{\Delta_{0}\Delta_1^4\Delta_2^3cyt}{C_3^2C_4}\frac{\Delta_{0}\Delta_1^3\Delta_2^4act}{C_3^2C_4}\frac{\Delta_{0}\Delta_1^3\Delta_2^4adt}{C_3C_1^2}\\
\frac{\Delta_{0}\Delta_1^3\Delta_2^4adx}{C_4C_1^2}\frac{\Delta_{0}\Delta_1^3\Delta_2^4bdx}{C_3C_4^2}\frac{\Delta_{0}\Delta_1^4\Delta_2^3bxz}{C_3C_4^2}\frac{\Delta_{0}\Delta_1^4\Delta_2^3byz}{C_4C_2^2}
\frac{\Delta_{0}\Delta_1^3\Delta_2^4bcy}{C_3C_2^2}\frac{\Delta_{0}\Delta_1^3\Delta_2^4acy}{C_3^2C_4}\frac{\Delta_{0}\Delta_1^4\Delta_2^3ayt}{C_3^2C_4}\frac{\Delta_{0}\Delta_1^4\Delta_2^3axt}{C_3C_1^2}\\
\frac{\Delta_{0}\Delta_1^3\Delta_2^4adt}{C_3C_1^2}\frac{\Delta_{0}\Delta_1^3\Delta_2^4act}{C_3^2C_4}\frac{\Delta_{0}\Delta_1^4\Delta_2^3cyt}{C_3^2C_4}\frac{\Delta_{0}\Delta_1^4\Delta_2^3xyz}{C_3C_2^2}
\frac{\Delta_{0}\Delta_1^3\Delta_2^4bcz}{C_4C_2^2}\frac{\Delta_{0}\Delta_1^3\Delta_2^4bdz}{C_3C_4^2}\frac{\Delta_{0}\Delta_1^4\Delta_2^3dxz}{C_3C_4^2}\frac{\Delta_{0}\Delta_1^4\Delta_2^3dxt}{C_4C_1^2}
\end{bmatrix}.
\notag\end{align}

Divide the weight matrix of $AD_{4n-4}(\wt_{d^{(4)}(N)})$ into $4n-3$ parts (by columns) as in Lemma \ref{mulstar}(a). Multiply all entries in the $(4k+1)$-th part by $\frac{C_1^2}{\Delta_0\Delta_1^4\Delta_2^4}$, for $k=0,\dots,n-1$; multiply all entries in the $(4k+2)$-th and $(4k+4)$-th parts by $\frac{C_3C_4}{\Delta_0\Delta_1^4\Delta_2^4}$, and multiply all entries in the $(4k+3)$-th part by $\frac{C_2^2}{\Delta_0\Delta_1^4\Delta_2^4}$, for $k=0,1,\dots,n-2$. We get the weight matrix of the weighted Aztec diamond $AD_{4n-4}(\wt_{\overline{N}})$, where
\[\overline{N}=\begin{bmatrix}
\frac{adg}{C_3\Delta_2}&\frac{ayt}{C_3\Delta_2}&\frac{acy}{C_3\Delta_1}&\frac{bcy}{C_3\Delta_1}&\frac{byz}{C_4\Delta_2}&\frac{bxz}{C_4\Delta_2}&\frac{bdx}{C_4\Delta_2}&\frac{adx}{C_4\Delta_2}\\
\frac{ceg}{C_4\Delta_2}&\frac{dxz}{C_4\Delta_2}&\frac{bdz}{C_4\Delta_1}&\frac{bcz}{C_4\Delta_1}&\frac{cyz}{C_3\Delta_2}&\frac{cyt}{C_3\Delta_2}&\frac{act}{C_3\Delta_2}&\frac{adt}{C_3\Delta_2}\\
\frac{ace}{C_4\Delta_1}&\frac{bdx}{C_4\Delta_1}&\frac{bxz}{C_4\Delta_2}&\frac{byz}{C_4\Delta_2}&\frac{bcy}{C_3\Delta_1}&\frac{acy}{C_3\Delta_1}&\frac{ayt}{C_3\Delta_2}&\frac{axt}{C_3\Delta_2}\\
\frac{acg}{C_3\Delta_1}&\frac{act}{C_3\Delta_1}&\frac{cyt}{C_3\Delta_2}&\frac{cyz}{C_3\Delta_2}&\frac{bcz}{C_4\Delta_1}&\frac{bdz}{C_4\Delta_1}&\frac{dxz}{C_4\Delta_2}&\frac{dxt}{C_4\Delta_2}
\end{bmatrix}.\]
Thus, we get from Lemma \ref{mulstar}
{\small
\begin{align}\label{star5}
\frac{\M(AD_{4n-4}(\wt_{\de^{(4)}(N)}))}{\M(AD_{4n-4}(\wt_{\overline{N}}))}=&\left(\Delta_0\Delta_1^4\Delta_2^4\right)^{(4n-4)(4n-3)}\left(\frac{1}{C_1}\right)^{(4n-4)2n}\notag\\
&\times\left(\frac{1}{C_3C_4}\right)^{(4n-4)2(n-1)}\left(\frac{1}{C_2}\right)^{(4n-4)2(n-1)}.
\end{align}}
Next, we divide the weight matrix of $AD_{4n-4}(\wt_{\overline{N}})$ into $(4n-3)\times(4n-4)$ blocks $M_{ij}$, for $i=1,2,\dots,4n-3$, and $j=1,2,\dots,4n-4$,  as in the Lemma \ref{matrixstar}. Multiply all entries of blocks $M_{2p+1,4q+1}$ and $M_{2p+2,4q+4}$ by $\dfrac{C_3}{at}$; multiply all entries of $M_{2p+1,4q+2}$ and $M_{2p+2,4q+3}$ by $\dfrac{C_3}{cy}$; multiply all entries of $M_{2p+1,4q+3}$ and $M_{2p+2,4q+2}$ by $\dfrac{C_4}{bz}$; and multiply all entries of $M_{2p+1,4q+4}$ and $M_{2p+2,4q+1}$ by $\dfrac{C_4}{dx}$. We get the weight matrix of the weighted Aztec diamond $AD_{4n-4}(wt_{\widetilde{N}})$, where
\[\widetilde{N}=\begin{bmatrix}
x/\Delta_2&y/\Delta_2&a/\Delta_1&b/\Delta_1&y/\Delta_2&x/\Delta_2&b/\Delta_1&a/\Delta_1\\
t/\Delta_2&z/\Delta_2&d/\Delta_1&c/\Delta_1&z/\Delta_2&t/\Delta_2&c/\Delta_1&d/\Delta_1\\
a/\Delta_1&b/\Delta_1&x/\Delta_2&y/\Delta_2&b/\Delta_1&a/\Delta_1&y/\Delta_2&x/\Delta_2\\
d/\Delta_1&c/\Delta_1&t/\Delta_2&z/\Delta_2&c/\Delta_1&d/\Delta_1&z/\Delta_2&t/\Delta_2
\end{bmatrix}.
\]
We apply Lemma \ref{matrixstar} and obtain
\begin{equation}\label{star6}
\frac{\M(AD_{4n-4}(\wt_{\overline{N}}))}{\M(AD_{4n-4}(\wt_{\widetilde{N}}))}=\left(\frac{at}{C_3}\right)^{4n-3}\left(\frac{dx}{C_4}\right)^{4n-3}\left(\frac{cy}{C_3}\right)^{4n-3}\left(\frac{bz}{C_4}\right)^{4n-3}.
\end{equation}
Using the equations (\ref{star1})--(\ref{star6}) above, we imply
\begin{equation} \label{star7}
\frac{\M(AD_{4n}(\wt_{N}))}{\M(AD_{4n-4}(\wt_{\widetilde{N}}))}=\Delta_0^{4n-2}\Delta_1^{8(n-1)^2}\Delta_2^{8(n-1)^2}C_1^{2n}C_2^{2n-2}
a^nb^{n-1}c^{n-1}d^nx^{n}y^{n-1}z^{n-1}t^{n}. \end{equation}
Moreover,  by applying the recurrence  (\ref{star7}) to $AD_{4n-4}(\wt_{\widetilde{N}})$, we get
\begin{align}\label{star8}
\frac{\M(AD_{4n-4}(\wt_{\widetilde{N}}))}{\M(AD_{4n-8}(\wt_{N}))}=&\left(\frac{\Delta_0}{\Delta_1\Delta_2}\right)^{4(n-1)-2}\left(\frac{1}{\Delta_2}\right)^{8(n-2)^2}
\left(\frac{1}{\Delta_1}\right)^{8(n-2)^2}\left(\frac{ad+xt}{\Delta_1^2\Delta_2^2}\right)^{2(n-1)}\notag\\
&\times\left(\frac{bc+yz}{\Delta_1^2\Delta_2^2}\right)^{2(n-1)-2}(x/\Delta_2)^{n-1}(y/\Delta_2)^{n-2}(z/\Delta_2)^{n-2}\notag\\
&\times(t/\Delta_2)^{n-1}(a/\Delta_1)^{n-1}(b/\Delta_1)^{n-2}(c/\Delta_1)^{n-2}(d/\Delta_1)^{n-1}. \end{align}
Finally, the first recurrence of the theorem follows from (\ref{star7}) and (\ref{star8}).
  \end{proof}

%%%%%%%%%%%%%%%%%%%%%%%%%%%%%%%%%%%%%%

\section{A variant of the triangular lattice}
\label{sec6}

We consider a new lattice obtained from the triangular lattice as follows. The triangular lattice can be partitioned into equilateral triangles of side-length 3, which we call \textit{basic triangles}. Each of these basic triangles consists of $9$ unit equilateral triangles. Remove all six lattice segments forming a unit hexagon around each vertex of these basic triangles. %Color the resulting lattice by black and white so that any two elementary regions that share an edge have opposite colors.
An elementary region in the new lattice is either a unit equilateral triangle or a unit rhombus (see Figure \ref{Fig30}). Each basic triangle on the new lattice is covered by three unit equilateral triangles and three unit rhombi. %The basic triangles come in two orientations: they may point upward or downward. %Without loss of generality, we can assume that the unit rhombus in each up-pointing basic triangle are black. We define a family of region on this new lattice as follows.

\begin{figure}\centering
\includegraphics[width=0.7\textwidth]{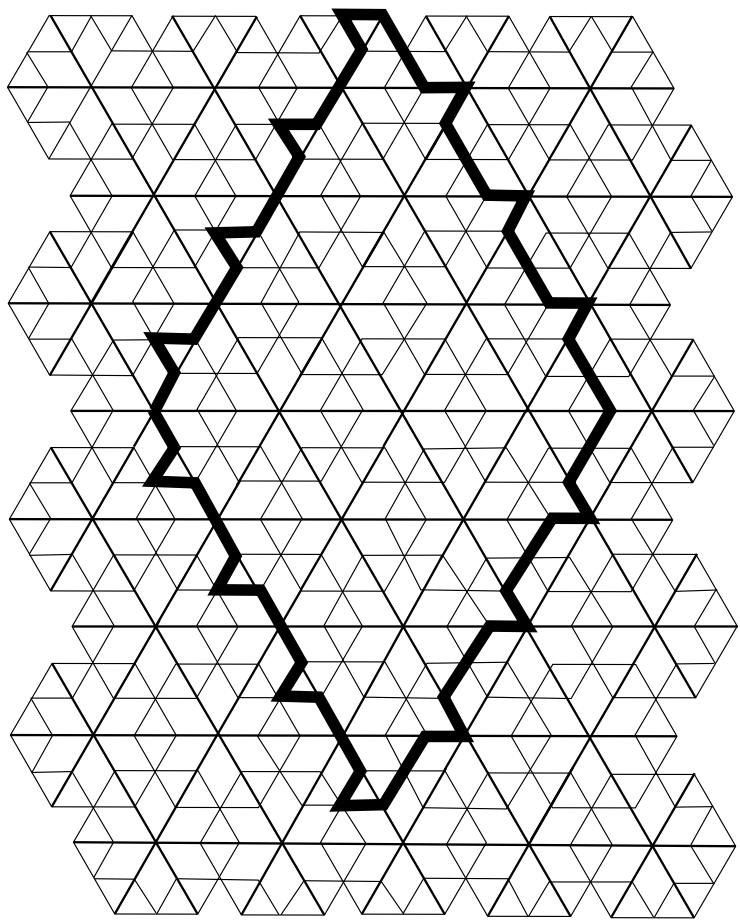}
\caption{The region $R_3$.}
\label{Fig30}
\end{figure}

Consider a vertical rhombus $R$ of side-lengths $3n+2$ whose four vertices are four lattice points and whose horizontal symmetry axis passing a vertex of some basic triangle. We consider a region that consists of all elementary regions lying completely or partly inside $R$, together with all unit triangles which have an edge resting on the boundary of $R$. Denote by $R_n$ the resulting region (see the region restricted by the bold contour in Figure \ref{Fig30}, for the case $n=3$). %Start from the western vertex of $R$, we draw a lattice path to its northern vertex so that for each step the elementary region on the right is black, and it has an edge resting on the boundary of $R$. The described path is the northwestern boundary of the regions. We draw similarly the northeastern boundary from the northern vertex to the eastern vertex of $R$,  with the one change that the elementary region on the right is now white for each step. The southeastern and southwestern boundaries are obtained from two boundaries above by reflecting them about the horizontal symmetric axis of $R$ (see Figure \ref{Fig30}, for the case $n=3$).  Denote by $R_n$ the  region bordered by four lattice paths above.
 The number of tilings of $R_n$ is given by the following theorem.

\begin{theorem}\label{tri} For any positive integer $n$
\begin{equation}\label{trieq}
\T(R_{n})=3^{n(n+1)}2^{(n+1)^2}.
\end{equation}
\end{theorem}

\begin{figure}\centering
\includegraphics[width=0.90\textwidth]{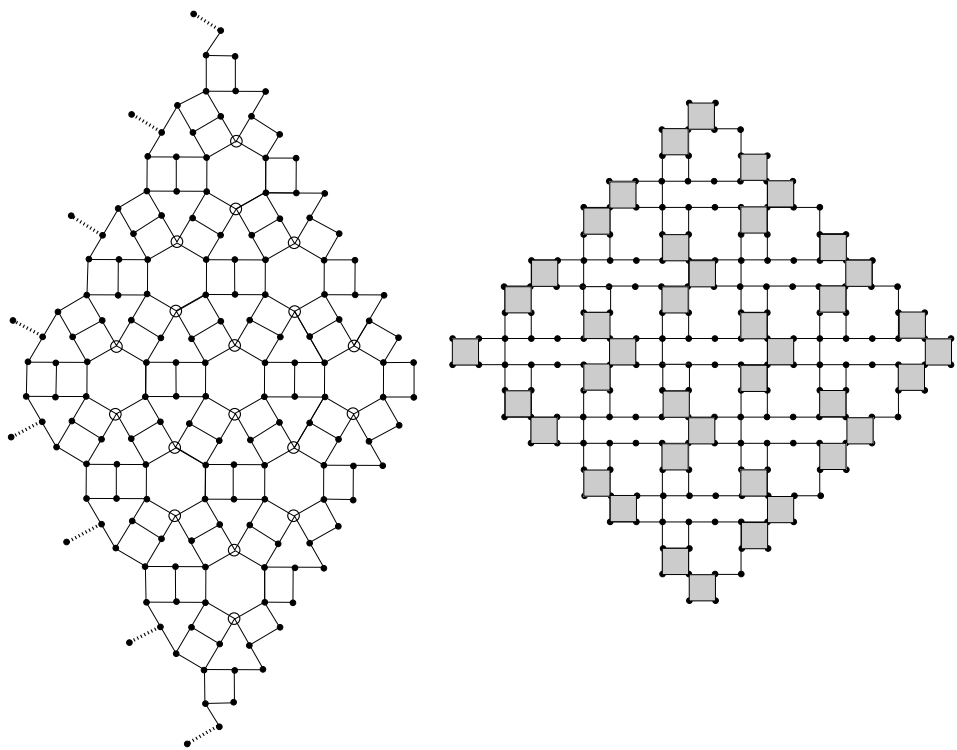}
\caption{Illustrating the transforming in the proof of Theorem \ref{tri}.}
\label{Fig31}
\end{figure}

\begin{proof}
Remove all edges incident to a vertex of degree $1$ in the dual graph of $R_n$, which are forced edges (illustrated in the left picture in Figure \ref{Fig31} by the dotted edges, for $n=3$). Then apply Vertex-splitting Lemma \ref{VS} to the top and bottom vertices of all regular hexagons in the resulting graph (illustrated by the circled vertices in left picture in Figure \ref{Fig31}).

Deform the resulting graph into a subgraph of the square lattice (see the  right picture in Figure \ref{Fig31}). Apply the vertex splitting Lemma \ref{VS} again to vertices of shaded cells which have even degree. Then apply Spider Lemma \ref{spider}(a) to all $3n^2+4n+1$ shaded cells in the resulting graph  (shown by the left picture in Figure \ref{Fig32}). Again, we remove all forced edges that are incident to a vertex of degree $1$. Finally, apply Edge-replacing Lemma \ref{ER} to all multiple edges arising from the previous steps  (see the right picture in Figure \ref{Fig32}). We get an isomorphic version of the weighted Aztec diamond $AD_{2n}(\wt_A)$ ($AD_{2n}(\wt_A)$ is obtained from the final graph by rotating $45^0$ clockwise and reflecting it about a horizontal line), where
\[
A=\begin{bmatrix}
1/2&1/2&1/2&1/2\\
1/2&1&0&1/2\\
3/2&0&1&1/2\\
3/2&3/2&1/2&1/2
\end{bmatrix}.
\]
By Lemmas \ref{VS}, \ref{ER} and \ref{spider}, we get
\begin{equation}\label{triang1}
\T(R_n)=2^{3n^2+4n+1}\M(AD_{2n}(\wt_A,)).
\end{equation}

\begin{figure}\centering
\includegraphics[width=0.90\textwidth]{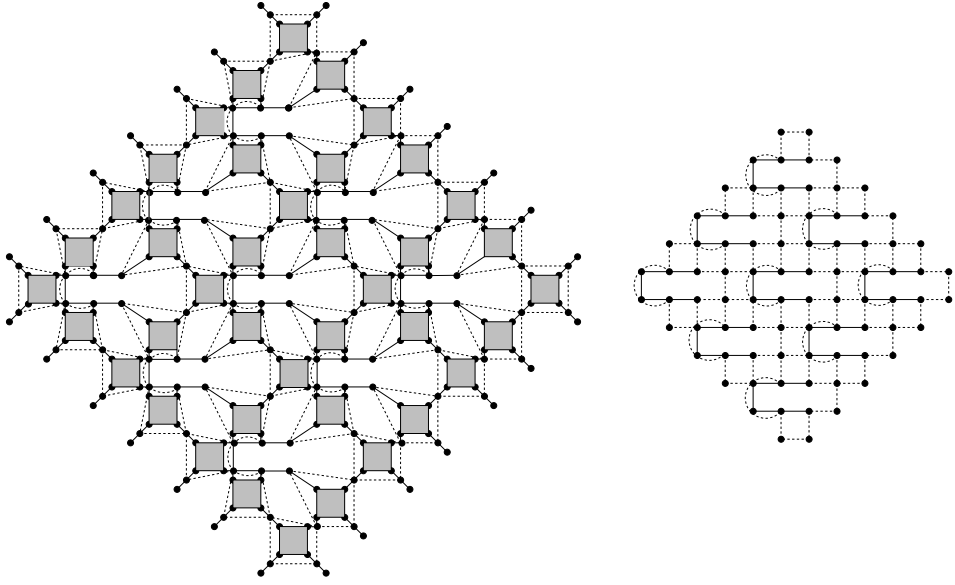}
\caption{Illustrating the transforming in the proof of Theorem \ref{tri} (cont.). The dotted edges are weighted by $1/2$, and the solid edges are weighted by $1$.}
\label{Fig32}
\end{figure}

Next, we calculate the values of $\M(AD_{2n}(\wt_A))$ by using Reduction Theorem.  It is easy to check that
\[\de(A)=\begin{bmatrix}
2/3&2&2&2/3\\
2/3&2/3&2/3&2/3\\
2/3&2/3&4/3&0\\
2/3&2&0&4/3
\end{bmatrix},
\]
\[\de^{(2)}(A)=\begin{bmatrix}
3/8&3/8&9/8&9/8\\
3/8&3/4&0&9/8\\
3/8&0&3/4&3/8\\
3/8&3/8&3/8&3/8
\end{bmatrix},\]
\[\de^{(3)}(A)=
\begin{bmatrix}
8/9&8/9&8/9&8/9\\
8/3&8/9&8/9&8/3\\
8/3&8/3&16/9&0\\
8/9&8/9&0&16/9
\end{bmatrix},\]
and
\[\de^{(4)}(A)=\frac{9}{16}\cdot A.\]
The cell-factors of cells in $AD_{2n}(\wt_A)$ are either $\frac{3}{4}$, $\frac{9}{4}$ or $\frac{1}{4}$, so by the Reduction Theorem
\begin{equation}\label{triang2}
\M(AD_{2n}(\wt_A))=\left(\frac{3}{4}\right)^{2n^2}\left(\frac{9}{4}\right)^{n^2}\left(\frac{1}{4}\right)^{n^2}\M(AD_{2n-1}(\wt_{\de(A)})).
\end{equation}
All cells in $AD_{2n-1}(\wt_{\de(A)})$ have the same cell-factor $\frac{16}{9}$, from the Reduction Theorem again, we have
\begin{equation}\label{triang3}
\M(AD_{2n}(\wt_A))=\left(\frac{16}{9}\right)^{(2n-1)^2}\M(AD_{2n-2}(\wt_{\de^{(2)}(A)})).
\end{equation}
The cells of the Aztec diamond $AD_{2n-2}(\wt_{d^{(2)}(A)})$ have cell-factors either $\frac{27}{64}$, $\frac{81}{64}$ or $\frac{9}{64}$. By the Reduction Theorem one more time, we obtain
\begin{equation}\label{triang4}
\M(AD_{2n-2}(\wt_{\de^{(2)}(A)}))=\left(\frac{27}{64}\right)^{2(n-1)^2}\left(\frac{81}{64}\right)^{(n-1)^2}\left(\frac{9}{64}\right)^{(n-1)^2}\M(AD_{2n-3}(\wt_{\de^{(3)}(A)})).
\end{equation}
Since the cell-factors of the cells in $AD_{2n-3}(\wt_{d^{(3)}(A)})$ are all $\frac{16^2}{9^2}$, the Reduction Theorem implies
\begin{equation}\label{triang5}
\M(AD_{2n-3}(\wt_{\de^{(3)}(A)}))=\left(\frac{16^2}{9^2}\right)^{(2n-3)^2}\M(AD_{2n-4}(\wt_{\de^{(4)}(A)})).
\end{equation}
From the fact $\de^{(4)}(A)=\frac{9}{16}A$,  Lemma \ref{mulstar}(a) implies
\begin{equation}\label{triang6}
\M(AD_{2n-4}(\wt_{\de^{(4)}(A)}))=\left(\frac{9}{16}\right)^{(2n-3)(2n-4)}\M(AD_{2n-4}(\wt_A)).
\end{equation}
By  (\ref{triang2})--(\ref{triang6}), we get
\begin{equation}\label{triang7}
\M(AD_{2n}(\wt_A))=2^{-8n+4}3^{4n-2}\M(AD_{2n-4}(\wt_A)).
\end{equation}
Repeated application of (\ref{triang7}) implies
\begin{equation}\label{trianglerecur}
\M(AD_{2n}(\wt_A))=3^{n(n+1)}2^{-2n(n+1)}.
\end{equation}
Finally, the equality (\ref{trieq}) is obtained from (\ref{triang1}) and (\ref{trianglerecur}).
  \end{proof}

%\begin{remark}
%One can get \begin{equation}\M(AD_{2n+1}(\wt_A))=3^{2n}2^{-4n+2}\M(AD_{2k}(\wt_A))\end{equation} by considering the forced edges in $AD_{2n+1}(\wt_A)$.
%\end{remark}

%%%%%%%%%%%%%%%%%%%%%%%%%%%%%%%%%%%%%%%%%%%%%%%%%%%%%%%
\subsection*{Acknowledgements}
Thanks to Professor Mihai Ciucu for his encouragement and many helpful discussions.

\end{document}